\documentclass[runningheads]{llncs}
\usepackage{tasks}
\usepackage[T1]{fontenc}
\usepackage{graphicx}
\usepackage{mathptmx}       
\usepackage{helvet}         
\usepackage{courier}        
\usepackage{type1cm}        
\usepackage{makeidx}         
\usepackage{graphicx}        
\usepackage[all]{xy}
\usepackage{multicol}        
\usepackage[bottom]{footmisc}
\usepackage{tikz}
\usepackage{tikz-cd}
\usetikzlibrary{arrows,automata}
\usetikzlibrary{shapes,snakes,backgrounds}
\usepackage{amsfonts}
\usepackage{amssymb}
\usepackage{amscd}
\usepackage{amstext}
\usepackage{amsmath}
\usepackage{pstricks}
\usepackage{enumerate}
\usepackage{hyperref}

\def\abs#1{|#1|}

\def\shuffle{\mathop{_{^{\sqcup\!\sqcup}}}}

\def\adots{\mathinner{\mkern2mu\raise1pt\hbox{.}
\mkern3mu\raise4pt\hbox{.}\mkern1mu\raise7pt\hbox{.}}}

\def\up#1{\raise 1ex\hbox{\footnotesize#1}}

\def\H{\mathcal{H}}

\def\span{\mathop\mathrm{span}\nolimits}

\def\path{\rightsquigarrow}

\def\Lyn{{\mathcal Lyn}}

\def\N{{\mathbb N}}
\def\C{{\mathbb C}}
\def\R{{\mathbb R}}
\def\Z{{\mathbb Z}}
\def\Q{{\mathbb Q}}

\def\Q{{\mathbb Q}}

\def\L{\mathrm{L}}
\def\H{\mathrm{H}}

\def\CONV{\mathrm{\tiny CONV }}
\def\gDIV{\mathrm{\tiny gDIV }}

\renewcommand{\Z}{{\mathbb Z}}
\renewcommand{\Q}{{\mathbb Q}}
\renewcommand{\R}{{\mathbb R}}
\renewcommand{\C}{{\mathbb C}}

\newcommand{\calD}{{\mathcal D}}
\newcommand{\calH}{{\mathcal H}}
\newcommand{\calZ}{{\mathcal Z}}
\newcommand{\calL}{{\mathcal L}}
\newcommand{\calM}{{\mathcal M}}
\newcommand{\calX}{{\mathcal X}}
\newcommand{\calU}{{\mathcal U}}

\newcommand{\Li}{\operatorname{Li}}

\def\QX{\mathbb{Q}\langle X \rangle}

\def\pol#1{\langle #1 \rangle}
\def\Lyn{\mathcal Lyn}

\def\calZ{\mathcal{Z}}

\def\L{\mathrm{L}}
\def\H{\mathrm{H}}

\def\abs#1{|#1|}
\def\scal#1#2{\langle #1\mid#2 \rangle}
\def\ncp#1#2{#1\langle #2\rangle}
\def\ncs#1#2{#1\langle \!\langle #2\rangle \!\rangle}
\def\path{\rightsquigarrow}
\def\pol#1{\langle #1 \rangle}

\def\shuffle{\mathop{_{^{\sqcup\!\sqcup}}}} 
\def\conc{\mathop{\tt conc}} 

\def\calH{{\mathcal H}}
\def\calL{{\mathcal L}}

\def\calP{{\mathcal P}}
\def\calQ{{\mathcal Q}}
\def\calR{{\mathcal R}}
\def\calX{{\mathcal X}}
\def\calZ{{\mathcal Z}}
\def\scal#1#2{\langle #1 | #2 \rangle}

\def\ncs#1#2{#1\langle\langle #2\rangle\rangle}

\def\ncp#1#2{#1\langle #2\rangle}

\def\Li{\mathrm{Li}}

\gdef\stuffle{\;%
  \setlength{\unitlength}{0.0125cm}%
  \begin{picture}(20,10)(220,580)
  \thinlines
  \put(220,592){\line( 0,-1){ 10}}
  \put(220,582){\line( 1, 0){ 20}}
  \put(240,582){\line( 0, 1){ 10}}
  \put(230,592){\line( 0,-1){ 10}}
  \put(225,587){\line( 1, 0){ 10}}
  \end{picture}\;
}
\def\pol#1{\langle #1 \rangle}

\newcommand\rsmraise[1]{%
  \ifx#1\displaystyle .8\else
    \ifx#1\textstyle .8\else
      \ifx#1\scriptstyle .6\else
        .45%
      \fi
    \fi
  \fi}
\begingroup
\def\2#1{\ifnum#1<10 0\fi\the#1}
\xdef\isodayandtime{
{\2\day-\2\month-\the\year\space\2{\count0}:%
\2{\count2}}}
\endgroup

\newcommand{\hooklongtwoheadrightarrow}{\lhook\joinrel\longtwoheadrightarrow}

\newcommand*{\longtwoheadrightarrow}{\ensuremath{\joinrel\relbar\joinrel\twoheadrightarrow}}

\makeatletter
\newcommand{\xhookdoubleheadrightarrow}[2][]{%
  \lhook\joinrel
  \ext@arrow 0359\rightarrowfill@ {#1}{#2}%
  \mathrel{\mspace{-15mu}}\rightarrow
}

\begingroup
\def\2#1{\ifnum#1<10 0\fi\the#1}
\xdef\isodayandtime{
{\2\day-\2\month-\the\year\space\2{\count0}:%
\2{\count2}}}
\endgroup

\newcounter{per1}
\definecolor{MyDarkBlue}{rgb}{0,0.08,0.4}

\begin{document}
\title{On the Algebraic Bases of Polyzetas}
%
%
\author{V. Hoang Ngoc Minh}
\authorrunning{V. Hoang Ngoc Minh}
%
\institute{University of Lille, 1 Place D\'eliot, 59024 Lille, France,
\email{vincel.hoang-ngoc-minh@univ-lille.fr}}
\maketitle              
\begin{abstract}
Two confluent rewriting systems in noncommutatives polynomials are constructed using the equations allowing the identification of the local coordinates (of second kind) of the graphs of the $\zeta$ polymorphism as being (shuffle or quasi-shuffle) characters and bridging two algebraic structures of polyzetas.
In each system, the left side of each rule corresponds to the leading monomial of the associated homogeneous in weight noncommutative polynomial while the right side is canonically represented on the $\Q$-algebra generated by irreducible terms:
\begin{itemize}
\item these polynomials, being totally lexicographically ordered, generate the kernels of the \textit{surjective} $\zeta$ polymorphism,

\item the images of these irreducible terms, by this $\zeta$ polymorphism, constitute $\Q$-algebraic bases of the $\Q$-algebra of polyzetas.
\end{itemize}
This means that the $\Q$-free algebra of polyzetas is graded and the irreducible polyzetas are transcendental numbers, $\Q$-algebraically independent, and then $\pi^2$ is $\Q$-algebraically independent on odd zeta values $\{\zeta(2q+1)\}_{1\le q\le 5}$ (so does $\pi$).

\end{abstract}

\section{Introduction}\label{introduction}
For any $r\ge1$ and $(s_1,\ldots,s_r)\in\N_{\ge1}$, for any $z\in\widetilde{\C\setminus\{0,1\}}$ and $n\ge1$, let
\begin{eqnarray}\label{def1def2}
\Li_{s_1,\ldots,s_r}(z):=\sum_{n_1>\ldots>n_r>0}\frac{z^{n_1}}{n_1^{s_1}\ldots n_r^{s_r}}&\mbox{and}&
\H_{s_1,\ldots,s_r}(n):=\sum_{n\ge n_1>\ldots>n_r>0}\frac1{n_1^{s_1}\ldots n_r^{s_r}}.
\end{eqnarray}
which are respectively called \textit{(multiple) polylogarithm} and\footnote{The harmonic sums are viewed as arithmetic function $\H_{s_1,\ldots,s_r}:\N\longmapsto\Q$. For any $n\ge0,r\ge1$ and $(s_1,\ldots,s_r)\in\N_{\ge1}$, the harmonic sums $\H_{s_1,\ldots,s_r}(n)$ is also called finte polyzeta.} \textit{(multiple) harmonic sum}. They lie by the following identity
\begin{eqnarray}\label{gs}
(1-z)^{-1}\Li_{s_1,\ldots,s_r}(z)=\sum_{n\ge1}\H_{s_1,\ldots,s_r}(n)z^n.
\end{eqnarray}
Let $\calH_r$ be $\{(s_1,\ldots,s_r)\in\N_{\ge1}^r,s_1>1\}$.
Then, for any $(s_1,\ldots,s_r)\in\calH_r$ the following real limits exist and coincid, by a Abel's theorem, with the so-called \textit{polyzetas}\footnote{Polyzeta is the contraction of polymorphism and of zeta (see Corollary \ref{zeta} below). See also Definition \ref{formalzetavalues} below, in which \textit{formal} and \textit{real} polyzetas are differentiated.} \cite{acta,VJM,CM}
\begin{eqnarray}\label{zetavalues}
\zeta(s_1,\ldots,s_r)
:=\lim_{{z\rightarrow1}}\Li_{s_1,\ldots,s_r}(z)
=\lim_{{n\rightarrow+\infty}}\H_{s_1,\ldots,s_r}(n)
=\sum_{n_1>\ldots>n_r>0}{n_1^{-s_1}\ldots n_r^{-s_r}}.
\end{eqnarray}
The numbers\footnote{If $s_1=1$ then $\zeta(s_1,\ldots,s_r)$ is divergent (see also \eqref{y_1^kw}--\eqref{y_1^k} and Example \ref{generalizedgamma} bellow).} $r$ and $s_1+\ldots+s_r$ are, respectively, the \textit{depth} and \textit{weight} of $\zeta(s_1,\ldots,s_r)$.

\goodbreak

These polyzetas involved as the coefficients of the monodromies, of the functional equations of polylogarithms (see \cite{FPSAC98,FPSAC99,SLC43}) and of the singular (resp. asymptotic) expansions (see \cite{Daresbury,JSC}) of polylogarithms (resp. hamonic sums).
 
With the conventions $\Li_{\emptyset}=1_{\widetilde{\C\setminus\{0,1\}}},\H_{\emptyset}=1$ and $\zeta(\emptyset)=1$, the objets introduced in \eqref{def1def2} and \eqref{zetavalues} do appear then three (total or partial) functions, \textit{that is} the functions $\Li_{\bullet},\H_{\bullet}$ and $\zeta$, on the free monoid $(({\mathbb N}_{\ge1})^*,\emptyset)$, generated by the alphabet ${\mathbb N}_{\ge1}$, with values in the subrings of the rings, respectively, of holomorphic functions on $\widetilde{\C\setminus\{0,1\}}$, rational-valued arithmetic functions and real numbers.

Euler earlier studied polyzetas \cite{euler2}, in classic analysis and in particular, the single zeta values $\{\zeta(s)\}_{s\ge1}$ and the double zeta values $\{\zeta(s_1,s_2)\}_{s_1>1,s_2\ge1}$. For the single zeta values, using the bernoulli numbers, $\{B_k\}_{k\ge0}$, he established the following identities \begin{eqnarray}\label{paire}
\frac{\zeta(2k)}{(2\pi)^{2k}}=\frac{B_{2k}}{2(2k)!}.
&\mbox{for}&k\ge1.
\end{eqnarray}

The rational ratios in \eqref{paire} and the transcendence of $\pi$ \cite{Lindemann} led to the trancendence of $\{\zeta(2k)\}_{k\ge1}$ (see \cite{Waldschmidt}). Ap\'ery proved that $\zeta(3)$ is irrational number \cite{apery}. In this same vain, Ball and Rivoal proved the irrationality of an infinite number of values of the $\zeta$ function at odd integers \cite{ballrivoal,rivoal} and Zudilin proved that at least one of the four numbers $\zeta(5),\zeta(7),\zeta(9)$ and $\zeta(11)$ is irrational \cite{zudilin}.

\begin{remark}\label{zeta(2)}
The linear integer relation $\pi^2=6\zeta(2)$, as a particular case of \eqref{paire}, allowed Euler to prove his earlier discovery of the usual relation, by numerical approximations, of the irrational numbers $\pi^2$ and $\zeta(2)$.
\end{remark}

For the double zeta values, Euler stated that $\zeta(6,2)$ can not be expressed on $\zeta(2),...,\allowbreak\zeta(8)$ and proved
\begin{eqnarray}\label{zeta(s,1)}
\zeta(2,1)=\zeta(3)
&\mbox{and}&
\zeta(s,1)=\frac12\Big(s\zeta(s+1)-\sum_{j=1}^{s-2}\zeta(j+1)\zeta(s-j)\Big),s>1.
\end{eqnarray}
He introduced the following identity (being an origin of the quasi-shuffle product,
denoted by $\stuffle$ and recursively defined in \eqref{stuffle} below \cite{euler2}
\begin{eqnarray}\label{Euler}
\zeta(s_1)\zeta(s_2)=\zeta(s_1,s_2)+\zeta(s_2,s_1)+\zeta(s_1+s_2).
\end{eqnarray}
In this same vain, Cartier also gave another identities concerning double zeta values using their commutative generating series \cite{Cartier_double}.

For any $s$ and $k\ge1$, denoting $\zeta(k)$ and $\zeta(s,\overbrace{1,\ldots,1}^{k\;\mathrm{times}})$ by $s_k$ and $c_{s,k}$, respectively, and hoping but doing not succeed to find a formula similar to \eqref{paire} for $\zeta(2k+1)$ (see Corollary \ref{transcendent} below), Nielsen extended \eqref{Euler} and, for any $n>0$ and $p,p'>1$, proved the following identities \cite{lewinlivre,nielsenlivre}
\begin{eqnarray}
s_{p'}s_{p}&=&s_{p'+p}+c_{p',p}+c_{p,p'},\label{reflexion}\\
c_{n,p}&=&(-1)^n\sum_{\nu=0}^{p-2}{n+\nu-1\choose n-1}c_{n+\nu,p-\nu}
+\sum_{\nu=0}^{n-2}(-1)^{\nu}{p+\nu-1\choose p-1}s_{n-\nu}s_{p+\nu}\nonumber\\
&&-(-1)^n{p+n-2\choose p-1}(s_{p+n}+c_{1,p+n-1}).\label{reduction}
\end{eqnarray}
He stated that $c_{n,p}$ is homogeneous polynomials of degree $n+k$ with rational coefficients on $\zeta(2),\ldots,\zeta(s+k)$ and deduced, for any $p>1$ and $p'\ge1$, that
\begin{eqnarray}\label{Nielsenrelations}
s_{p+1}=\sum_{\nu=1}^{p-1}c_{\nu,p-\nu+1}&\mbox{and}&
c_{1,p'}=\frac{p'}2s_{p'+1}-\frac12\sum_{\nu=2}^{p'-1}s_{\nu}s_{p'+1-\nu}.
\end{eqnarray}
He earlier used iterated integrals \cite{chen} (see \eqref{iteratedintegral} below) to represent the following particular polylogarithms
\begin{eqnarray}
\Li_1(z)=\log\frac1{1-z}&\mbox{and}&\Li_s(z)=\int_0^z\omega_0(t)\Li_{s-1}(t),\\
\Li_{0,\underbrace{1,\ldots,1}_{k\;\mathrm{times}}}(z)=\frac1{k!}\log^k\frac1{1-z}
&\mbox{and}&\Li_{s,\underbrace{1,\ldots,1}_{k\;\mathrm{times}}}(z)
=\int_0^z\omega_0(t)\Li_{s-1,\underbrace{1,\ldots,1}_{k\;\mathrm{times}}}(t),
\end{eqnarray}
where $\omega_0$ and $\omega_1$ are the following differential $1$-forms
\begin{eqnarray}\label{differentialforms}
\omega_0(z)=z^{-1}{dz}&\mbox{and}&\omega_1(z)=(1-z)^{-1}{dz},
\end{eqnarray}

More generaly, the polylogarithm in \eqref{def1def2} can be also defined as an iterated integral, of $\omega_0$ and $\omega_1$ introduced in \eqref{differentialforms} and along the path $0\path z$ over $\widetilde{\C\setminus\{0,1\}}$:
\begin{eqnarray}\label{iteratedintegral}
\Li_{s_1,\ldots,s_r}(z)=\left\{\begin{array}{lcl}
\displaystyle\int_0^z\omega_1(s)\Li_{s_2,\ldots,s_r}(s)&\mbox{if}&s_1=1,\cr
\displaystyle\int_0^z\omega_0(s)\Li_{s_1-1,\ldots,s_r}(s)&\mbox{if}&s_1>1.
\end{array}\right.
\end{eqnarray}
These iterated integrals satisfies the shuffle product \cite{chen}, denoted by $\shuffle$ and recursively defined in \eqref{shuffle} below. It follows also that
\begin{eqnarray}\label{differentiation}
\frac{d}{dz}\Li_{s_1,\ldots,s_r}(z)=\left\{\begin{array}{lcl}
\omega_1(z)\Li_{s_2,\ldots,s_r}(z)&\mbox{if}&s_1=1,\cr
\omega_0(z)\Li_{s_1-1,\ldots,s_r}(z)&\mbox{if}&s_1>1.
\end{array}\right.
\end{eqnarray}

The $\{\zeta(s_1,\ldots,s_r)\}^{r\ge1}_{s_1>1,s_2,\ldots,s_r\ge1}$ are also called \textit{multi zeta values} (MZV for short) \cite{zagier} or \textit{multiple harmonic series} \cite{Hoffman,hoffman2} or \textit{Euler-Zagier sums} \cite{borwein4}. One can found in their biographies some recent applications of these special values in algebraic geometry, knots invariants of Vassiliev-Kontsevich, modular forms, quantum electrodynamic, \ldots.

More recently, extending the Euler's numerical appoximation method, mentioned previously in Remark \ref{zeta(2)}, many new linear integer relations for polyzetas are also detected in hight performance computing \cite{blumlein,borwein2,borwein4} using the algorithms of\footnote{LLL is a contraction of A. Lenstra, H. Lenstra and L. Lov\'asz.} LLL type and the truncations at order $n$ of $\{\zeta(s_1,\ldots,s_r)\}^{r\ge1}_{s_1>1,s_2,\ldots,s_r\ge1}$, \textit{i.e.} $\{\H_{s_1,\ldots,s_r}(n)\}^{r\ge1}_{s_1>1,s_2,\ldots,s_r\ge1}$ satisfying the following difference equation and satisfying the $\stuffle$ product as already said above (see also Theorem \ref{structure1} below)
\begin{eqnarray}
\H_{s_1,\ldots,s_r}(n)-\H_{s_1,\ldots,s_r}(n-1)={n^{-s_1}}.
\end{eqnarray}

In this numerical appoximation approach, the main problem is to detect with near certainty which polyzetas can not be expressed on $\{\zeta(2),\ldots,\zeta(s+k)\}$ and are qualified as \textit{new constants} \cite{borwein2,borwein4} (as for the
Euler's result on $\zeta(6,2)$). Then such polyzetas could be $\Q$-algebraically
independent on $\{\zeta(2),\ldots,\zeta(s+k)\}$ (also see Example \ref{E} below) and these polyzetas could be transcendental numbers (see the proof of Corollary
\ref{transcendent} below).

Checking $\Q$-linear relations on $\{\zeta(s_1,\ldots,s_r)\}^{r\ge1}_{s_1>1,s_2,\ldots,s_r\ge1\atop2\le s_1+\ldots+s_r\le12}$ by LLL type programs, Zagier states that the $\Q$-module generated by MZV is apparently graded (see Theorem \ref{graded} below for a proof) and then guess

\begin{conjecture}[\cite{zagier}]\label{Zconj}
For any $k\ge1$, let
\begin{eqnarray*}
\calZ_k:=\mathrm{span}_{\Q}\{\zeta(s_1,\ldots,s_r)\}^{r\ge1}_{s_1>1,s_2,\ldots,s_r\ge1\atop s_1+\ldots+s_r=k}&\mbox{and}&d_k:=\dim\calZ_k.
\end{eqnarray*}
Then $\{d_k\}_{k\ge0}$ satistfies the recursion\footnote{See also \cite{racinet,FPSAC97,kaneko} for other algebraic checks.} $d_k=d_{k-3}+d_{k-2}$ ($k\ge4$) and $d_1=0,d_2=d_3=1$.
\end{conjecture}

These conjecture then each module $\calZ_k$ is $\Q$-linearly generated by $d_k$ transcendental zeta values and there is no $\Q$-linear relation among elements of different $\calZ_k$. In this same vain, intensive tests, via linear relations among  the convergent polyzetas, are also done in \cite{blumlein,ecalle,flajoletsalvy,Furusho,racinet,kaneko}
and several estimations of these dimensions, via \textit{motivic} structures, are also obtained by Deligne and Goncharov \cite{delignegoncharov}, Terasoma \cite{Terasoma}. Even better, with the \textit{motivic} polyzetas, in \cite{brown}, Brown showed that the polyzetas indexed by $2$ and $3$ are generators for the $\Q$-vector space of the polyzetas, \textit{i.e.}
\begin{eqnarray}\label{Image}
\calZ:=\mathrm{span}_{\Q}\{\zeta(s_1,\ldots,s_r)\}^{r\ge1}_{s_1>1,s_2,\ldots,s_r\ge1}.
\end{eqnarray}

As already said, the analytic properties, as monodromies and functional equations of polylogarithms, are studied in \cite{FPSAC98,FPSAC99,SLC43}, singular (resp. asymptotic) expansions of polylogarithms (resp. harmonic sums), for $z\to1$ (resp. $n\to+\infty$) in the comparison scale $\{(1-z)^a\log^b((1-z)^{-1})\}_{a\in\Z\atop b\in\N}$ (resp. $\{n^a\log^b(n)\}_{a\in\Z\atop b\in\N}$ or $\{n^a\H_1^b(n)\}_{a\in\Z\atop b\in\N}$), are realized in \cite{Daresbury,JSC} leading to the constants (see also the finite parts in \eqref{gamma_bullet} and \eqref{regularization1}-\eqref{regularization2} below)\footnote{The coefficients of these expansions can be different if one changes the comparison scales.} given in \eqref{zetavalues}. Now, studying Conjecture \ref{Zconj}, in continuation with the works carried out in \cite{Bui,JSC,FPSAC97} and in a symbolic approach (as with the manipulation of the symbols of the polyzetas appearing in \eqref{zetavalues}--\eqref{Nielsenrelations} and being typped as the real numbers), this work provides more explanations regarding the algorithm {\bf LocalCoordinateIdentification}, partially implemented in \cite{Bui} and briefly described in \cite{LT12,Tokyo} as an application of an Abel type theorem \cite{Daresbury,JSC} (see Theorem \ref{renormalization1} below) concerning the noncommutative generating series (see \eqref{Zshuffle}--\eqref{Zstuffle} below) of $\{\Li_{s_1,\ldots,s_r}\}^{r\ge1}_{s_1,\ldots,s_r\ge1}$ (resp. $\{\H_{s_1,\ldots,s_r}\}^{r\ge1}_{s_1,\ldots,s_r\ge1}$), over $X=\{x_0,x_1\}$ and $Y=\{y_k\}_{k\ge1}$ generating the free monoids $(X^*,1_{X^*})$ and $(Y^*,1_{Y^*})$, with respect to the concatenation (denoted by $\tt conc$ and omitted if there is no ambiguity \cite{berstel,lothaire}).
This Abel theorem exploits the following one-to-one correspondences
\begin{eqnarray}\label{piXY1}
x_0^{{s_1}-1}x_1\ldots x_0^{{s_r}-1}x_1\in X^*x_1
\mathop{\rightleftharpoons}\limits_{\pi_X}^{\pi_Y}y_{s_1}\ldots y_{s_r}\in Y^*
\leftrightarrow(s_1,\ldots,s_r)\in\N_{\ge1}^*,
\end{eqnarray}
where $\pi_X$ and $\pi_Y$ are the following morphisms of monoids
\begin{eqnarray}
\left\{\begin{array}{rcl}
\pi_X:(Y^*,1_{Y^*})&\longrightarrow&(X^*x_1,1_{X^*})\cr
y_k&\longmapsto&x_0^{k-1}x_1
\end{array}\right\}
\mbox{ and }
\left\{\begin{array}{rcl}
\pi_Y:(X^*x_1,1_{X^*})&\longrightarrow&(Y^*,1_{Y^*})\cr
x_0^{k-1}x_1&\longmapsto&y_k\label{piXY2}
\end{array}\right\}.
\end{eqnarray}
It uses in practice the indexations of polylogarithms (resp. harmonic sums) by words over $X$ (resp. $Y$) meaning that \cite{CM} (see Theorem \ref{structure1} below)
\begin{enumerate}
\item $\Li_{\bullet}$ (resp. $\H_{\bullet}$) is also a function on the monoid $(X^*,1_{X^*})$ (resp. $(Y^*,1_{Y^*})$) with values in the subrings of the ring of holomorphic functions on $\widetilde{\C\setminus\{0,1\}}$ (resp. rational-valued arithmetic functions), \textit{i.e.} one puts $\Li_{1_{Y^*}}=1_{\widetilde{\C\setminus\{0,1\}}}$ (resp. $\H_{1_{Y^*}}=1$), $\Li_{x_0}(z)=\log(z)$
and, for any $x_0^{s_1-1}x_1\ldots x_0^{s_r-1}x_1\in X^*x_1$ (resp. $y_{s_1}\ldots y_{s_r}\in Y^*$),
\begin{eqnarray}\label{indexation0}
\Li_{x_0^{s_1-1}x_1\ldots x_0^{s_r-1}x_1}=\Li_{s_1,\ldots,s_r}
&(\mbox{resp.}&\H_{y_{s_1}\ldots y_{s_r}}=\H_{s_1,\ldots,s_r}).
\end{eqnarray}

\item $\zeta$ becomes then a partial function on $X^*$ and $Y^*$, with values in the subring of the ring of real numbers (see Corollary \ref{zeta} below), \textit{i.e.} $\zeta(1_{X^*})=\zeta(1_{Y^*})=1$ and, for any $x_0^{s_1-1}x_1\ldots x_0^{s_r-1}x_1\in x_0X^*x_1$ and $y_{s_1}\ldots y_{s_r}\in (Y\setminus\{y_1\})Y^*$,  as the real limits in \eqref{zetavalues},
\begin{eqnarray}\label{indexation}
\left\{\begin{array}{rcl}
\zeta(x_0^{s_1-1}x_1\ldots x_0^{s_r-1}x_1)&:=&
\lim\limits_{z\to1}\Li_{x_0^{s_1-1}x_1\ldots x_0^{s_r-1}x_1}(z)\\
\zeta(y_{s_1}\ldots y_{s_r})&:=&
\lim\limits_{n\to\infty}\H_{y_{s_1}\ldots y_{s_r}}(n)
\end{array}\right\}
=\zeta(s_1,\ldots,s_r).
\end{eqnarray}
\end{enumerate}

These functions on monoids lead to consider the isomorphisms $\Li_{\bullet}$ and $\H_{\bullet}$ \cite{CM} (see Theorem \ref{structure1} below) and then, as a consequence, the $\zeta$ polymorphism partially defined \cite{CM} (see Corollary \ref{zeta} below), both from the $\Q$-algebras of noncommutative polynomials, respectively, over $X$ and $Y$, equipped the shuffle product (denoted by $\shuffle$) and the quasi-shuffle (denoted by $\stuffle$) product, \textit{i.e.}
\begin{eqnarray}\label{polynomials}
(\ncp{\Q}{X},\shuffle,1_{X^*})&\mbox{and}&
(\ncp{\Q}{Y},\stuffle,1_{Y^*}),
\end{eqnarray}
with values, respectively, in the $\Q$-algebras of polylogarithms and harmonic sums, \textit{i.e.}
\begin{eqnarray}\label{Qalgebas}
(\mathrm{span}_{\Q}\{\Li_w\}_{w\in X^*},\times,1_{\widetilde{\C\setminus\{0,1\}}})&\mbox{and}&(\mathrm{span}_{\Q}\{\H_w\}_{w\in Y^*},\times,1).
\end{eqnarray}
and in the $\Q$-algebra of the polyzetas, \textit{i.e.} $(\calZ,\times,1)$ being a subalgebra of $(\R,\times,1)$.

The $\Q$-algebras in \eqref{polynomials} admit totally ordered pure transcendence bases, as the sets of Lyndon words $\Lyn X$ and $\Lyn Y$ (or $\{S_l\}_{l\in\Lyn X}$ \cite{lothaire,reutenauer} and $\{\Sigma_l\}_{l\in\Lyn Y}$ \cite{VJM,CM}), for which various $\Q$-algebraic bases for $(\calZ,\times,1)$ will be constructed (see Example \ref{E1} and Proposition \ref{directsum} below). Moreover, the $\Q$-algebraic bases $\{S_l\}_{l\in\Lyn X}$ and $\{\Sigma_l\}_{l\in\Lyn Y}$ admit the dual bases, \textit{i.e.} $\{P_l\}_{l\in\Lyn X}$ and $\{\Pi_l\}_{l\in\Lyn Y}$, being totally ordered bases of the Lie algebras of the primitive elements, of the $\Q$-bialgbras respectively (see \eqref{Hopfshuffle} and \eqref{Hopfstuffle}, Examples \ref{P_lS_l}--\ref{Pi_lSigma_l} below)
\begin{eqnarray}\label{bialgbras}
(\ncp{\Q}{Y},\conc,1_{Y^*},\Delta_{\stuffle})&\mbox{and}&
(\ncp{\Q}{X},\conc,1_{X^*},\Delta_{\shuffle}).
\end{eqnarray}

The equations on the graphs of this $\zeta$ polymorphism, which are the following grouplike series for the coproducts $\Delta_{\shuffle}$ and $\Delta_{\stuffle}$ (see \eqref{Zshuffle}--\eqref{Zgamma} below)\footnote{$\gamma$ is the Euler's gamma constant. See also Remark \ref{associator} below.} \cite{acta,VJM,CM}
\begin{eqnarray}\label{gplkseries}
Z_{\shuffle}=\prod_{l\in\Lyn X\setminus X}^{\searrow}e^{\zeta(S_l)P_l},
&Z_{\gamma}=e^{\gamma y_1}Z_{\stuffle},
&Z_{\stuffle}=
\prod_{l\in\Lyn Y\setminus\{y_1\}}^{\searrow}e^{\zeta(\Sigma_l)\Pi_l},
\end{eqnarray}
in which $\{\zeta(S_l)\}_{l\in\Lyn X\setminus X},\{\gamma\}\cup\{\{\zeta(\Sigma_l)\}_{l\in\Lyn Y\setminus\{y_1\}}\}$ and $\{\zeta(\Sigma_l)\}_{l\in\Lyn Y\setminus\{y_1\}}$ are their respective locale coordinates of second kind on the group of grouplike series, belonging to the $\Q$-algebra $(\calZ,\times,1)$ being a sub algebra of $(\R,\times,1)$. These equations, \textit{i.e.}
\begin{eqnarray}\label{bridge}
Z_{\gamma}=e^{\gamma y_1-\sum\limits_{k\ge2}\zeta(k){(-y_1)^k}/k}\pi_Y(Z_{\shuffle})&\mbox{and}&
Z_{\stuffle}=e^{-\sum\limits_{k\ge2}{\zeta(k)}{(-y_1)^k}/k}\pi_Y(Z_{\shuffle}),
\end{eqnarray}
bridge the algebraic structures of the polyzetas (see Corollaries \ref{zeta} below) and provide the kernels of the $\zeta$ polymorphism \cite{acta,VJM,CM} which are the following ideals of the $\Q$-algebras of noncommutative polynomials in \eqref{polynomials} (see Corollary \ref{pont} below)
\begin{eqnarray}
\shuffle-\mbox{ideal }\calR_X&\mbox{and}&\stuffle-\mbox{ideal }\calR_Y.
\end{eqnarray}

These kernels are obtained by identifying the local coordinates in the bridge equations (see \eqref{bridge}) providing the $\Q$-algebraic relations\footnote{See Examples \ref{Pi_lSigma_lsuite}--\ref{P_lS_lsuite} and \ref{E1} below for some expressions of the polyzetas $\{\zeta(S_l)\}_{l\in\Lyn X\setminus X}$ and $\{\zeta(\Sigma_w)\}_{w\in(Y\setminus\{y_1\})Y^*}$ and their $\Q$-algebraic relations.} among $\{\zeta(S_l)\}_{l\Lyn X\setminus X}$ and among $\{\zeta(\Sigma_l)\}_{l\Lyn Y\setminus{y_1}}$ (\textit{i.e.} the coordinates of the grouplike series $\{Z_{\shuffle},Z_{\gamma},Z_{\stuffle}\}$ satistfy $\Q$-algebraic relations) whereas the algorithm presented in \cite{FPSAC97}, based on the so-called \textit{regularized double shuffle relations} (recalled in Remark \ref{regularizations} below) provides the $\Q$-algebraic relations among $\{\zeta(l)\}_{l\Lyn X\setminus X}$, or equivalently $\{\zeta(l)\}_{l\Lyn Y\setminus{y_1}}$. These kernels provides then homogeneous in weight polynomials\footnote{See Example \ref{E3} below for some homogeneous in weight polynomials generating inside $\ker\zeta$.}, belonging to $\calQ_X$ or to $\calQ_Y$, constituting the ideals $\calR_X$ and $\calR_Y$ (see \eqref{Q_XQ_Y}--\eqref{R_XR_Y} below) and being viewed as confluent rewriting systems, in which each rewriting rule\footnote{See Example \ref{E0} below for some rewriting rules.} belonging to $\calR^X_{irr}$ or to $\calR^Y_{irr}$ (see \eqref{Rirr} below) admits
\begin{enumerate}
\item the left side as being the leading term of the associated homogeneous polynomial,

\item the right side as being canonically represented on the graded $\Q$-algebra generated by irreducible terms (see also Remark \ref{bijection} and Definition \ref{formalzetavalues} below)\footnote{See Example \ref{E0} below for some elements of $\calL^{X,\infty}_{irr}$ or to $\calL^{Y,\infty}_{irr}$.}, belonging to
\begin{eqnarray}\label{formalirre}
\calL^{X,\infty}_{irr}\subsetneq\{S_l\}_{l\Lyn X}&\mbox{or to}&
\calL^{Y,\infty}_{irr}\subsetneq\{\Sigma_l\}_{l\in\Lyn Y},
\end{eqnarray}
which are totally ordered and encode the algebraic generators of the $\Q$-free algebra of polyzetas (see Proposition \ref{directsum} and Theorem \ref{graded} below) \cite{acta,VJM,CM}.
\end{enumerate}

It follows that the irreducible polyzetas belonging to (see \eqref{restriction}--\eqref{Lirr} below)\footnote{See Example \ref{E1} below for some elements of $\calZ^{X,\infty}_{irr}$ or to $\calZ^{Y,\infty}_{irr}$.}
\begin{eqnarray}
\calZ^{X,\infty}_{irr}=\{\zeta(p)\}_{p\in\calL^{X,\infty}_{irr}}
&\mbox{or to}&
\calZ^{Y,\infty}_{irr}=\{\zeta(p)\}_{p\in\calL^{Y,\infty}_{irr}},
\end{eqnarray}
are $\Q$-algebraically free and are transcendental numbers (see Corollary \ref{transcendent} below), for which the elements of $\calL^{X,\infty}_{irr}$ and $\calL^{Y,\infty}_{irr}$ are viewed as the \textit{formal} irreducible polyzetas (see Definition \ref{formalzetavalues} below). Moreover, $\calL^{X,\infty}_{irr}$ and $\calL^{Y,\infty}_{irr}$ are images of $\calZ^{X,\infty}_{irr}$ and $\calZ^{Y,\infty}_{irr}$ by a section of \textit{surjective} $\zeta$ polymorphism (described in Corollary \ref{zeta} below) such that the following \textit{restrictions} of the $\zeta$ polymorphism are \textit{bijective} (see \eqref{restriction} below)
\begin{eqnarray}\label{bijective}
\zeta:(\Q[\calL^{X,\infty}_{irr}],\shuffle,1_{X^*})&\hooklongtwoheadrightarrow&(\Q[\calZ^{X,\infty}_{irr}],\times,1),\\
\zeta:(\Q[\calL^{Y,\infty}_{irr}],\stuffle,1_{Y^*})&\hooklongtwoheadrightarrow&(\Q[\calZ^{Y,\infty}_{irr}],\times,1).
\end{eqnarray}

By \eqref{formalirre}, the sets $\calL^{X,\infty}_{irr}$ and $\calL^{Y,\infty}_{irr}$, of irreducible terms, are $\Q$-algebraicly free and then, by the bijections in \eqref{bijective}, the sets $\calZ^{X,\infty}_{irr}$ and $\calZ^{Y,\infty}_{irr}$, of irreducible polyzetas, constitute algebraic bases for the $\Q$-algebra of polyzetas $(\calZ,\times,1)$ being a subalgebra of $(\R,\times,1)$ (see also Definition \ref{formalzetavalues}).

Therefore, through Example \ref{E} below up to weight $12$,
\begin{enumerate}
\item $\pi$ is $\Q$-algebraically free on the odd zeta values\footnote{\label{ratio}The readers are also invited to have a look at \cite{PMB} for a proof, with the present symbolic approach, of certains rational ratios $\zeta(s_1,\cdots,s_r)/\pi^{s_1+\cdots+s_r}$, proving the transcendence of the corresponding $\zeta(s_1,\cdots,s_r)$ using the transcendence of $\pi$ \cite{Lindemann}.} $\{\zeta(2q+1)\}_{1\le q\le 5}$ (see also \cite{VJM}).

\item Conjecture \ref{Zconj} held (see also \cite{racinet,FPSAC97,kaneko}) and the odd zeta values $\{\zeta(3),\zeta(5),\zeta(7),\allowbreak\zeta(9),\zeta(11)\}$ are irreducible then are $\Q$-algebraicly independent (see also \cite{FPSAC97}).
\end{enumerate}

The organization of this paper is follows
\begin{enumerate}
\item[(a)] In Section \ref{Combinatorics}, the isomorphy between the following shuffle and quasi-shuffle $\Q$-bialgebras will be recalled (Theorem \ref{isomorphy} below)
\begin{eqnarray}
(\ncp{\Q}{Y},\conc,1_{Y^*},\Delta_{\stuffle})&\mbox{and}&
(\ncp{\Q}{Y},\conc,1_{Y^*},\Delta_{\shuffle}),
\end{eqnarray}
for which, by the Cartier-Quillen-Milnor-Moore theorem (CQMM  for short) and Poincar\'e-Birkhoff-Witt theorem (PBW for short), various dual bases are constructed, in \eqref{P_w} and \eqref{Pi_w} below, for factorizing in the M\'elan\c{c}on-Reute\-nauer-Sch\"utzen\-berger form (MRS for short), the diagonal series $\calD_X$ and $\calD_Y$ (see \eqref{diagonalX} and \eqref{diagonalY} below)
\begin{eqnarray}\label{diagonal}
\calD_X:=\sum_{w\in X^*}w\otimes w=\prod_{l\in\Lyn X}^{\searrow}e^{S_l\otimes P_l}
&\mbox{and}&
\calD_Y:=\sum_{w\in Y^*}w\otimes w=\prod_{l\in\Lyn Y}^{\searrow}e^{\Sigma_l\otimes\Pi_l}.
\end{eqnarray}

\item[(b)] In Section \ref{Indexation}, the isomorphies of the $\Q$-algebras $\Li_{\bullet}$ and $\H_{\bullet}$, described in \eqref{def1def2}--\eqref{gs}, from the shuffle and quasi-shuffle $\Q$-algebras to the $\Q$-algebras of polylogarithms and harmonic sums (see \eqref{polynomials}--\eqref{Qalgebas}) will be explained (Theorem \ref{structure1}).

On the one hand, these isomorphies algebraically put, by \eqref{diagonal}, the noncommutative generating series of polylogarithms and of harmonic sums in the MRS form (see \eqref{NCGS}--\eqref{Zstuffle} below)
\begin{eqnarray}
\L=(\Li_{\bullet}\otimes\mathrm{Id})\calD_X=\prod_{l\in\Lyn }^{\searrow}e^{\Li_{S_l}P_l}
&\mbox{and}&
\H=(\H_{\bullet}\otimes\mathrm{Id})\calD_Y=\prod_{l\in\Lyn Y}^{\searrow}e^{\H_{\Sigma_l}\Pi_l}
\end{eqnarray}
and, on the other hand, analytically establish an Abel like theorem (Theorem \ref{renormalization1}) and identities bridging algebraic structures of poyzetas (Corollary \ref{pont}).

\item[(c)] In Section \ref{Algorithm}, the algorithm {\bf LocalCoordinateIdentification}, bringing more achievements than in \cite{Bui}, will provide two confluent rewriting systems\footnote{See also Remarks \ref{ordered1}--\ref{ordered2} below.} (see \eqref{Rirr} below)
\begin{eqnarray}
(\Q1_{Y^*}\oplus(Y\setminus\{y_1\})\ncp{\Q}{Y},\calR^Y_{irr})&\mbox{and}&(\Q1_{X^*}\oplus x_0\ncp{\Q}{X}x_1,\calR^X_{irr})
\end{eqnarray}
(without critical pairs and admitting, respectively, $\calL^{Y,\infty}_{irr}$ and $\calL^{X,\infty}_{irr}$ as the sets of irreducible terms).

These rewriting systems provide also the image and the kernels of the $\zeta$ polymorphism (Proposition \ref{directsum}) and then the algebraic structures and the arithmetical natures of polyzetas (Theorem \ref{graded} and Corollaries \ref{transcendent}--\ref{independence}).
\end{enumerate}

\section{Algebraic combinatorial framework}\label{Combinatorics}
Throughout this work, coefficients belong to a commutative ring\footnote{although
some of the properties already hold for general commutative semiring \cite{berstel}.} $A$ containing $\Q$ and, unless explicitly stated, all tensor products will be considered over the ambient ring (or field).

For all matters concerning finite or infinite alphabets, $X$ ($=\{x_0,x_1\}$ and similar) or $Y$ ($=\{y_k\}_{k\ge1}$ and similar), we use a generic model noted $\calX$ in order to state their common combinatorial features (see also Theorem \ref{isomorphy} below).

In the Section \ref{Indexation}, to manipulate convergent and divergent indexed by words polyzetas (see \eqref{indexation}), the following notations will be also adapted, for greater convenience,
\begin{eqnarray}\label{convenience}
\begin{array}{llrllll}
\mbox{if}&\calX=X&\mbox{then let}&\gDIV:=X&\mbox{and}&\CONV:=x_0X^*x_1\\
&&\mbox{else let}&\gDIV:=\{y_1\}&\mbox{and}&\CONV:=(Y\setminus\{y_1\})Y^*.
\end{array}
\end{eqnarray}
It follows that $\Lyn\calX\setminus\gDIV\subset\CONV$.

Once $\calX$ has been totally ordered ($x_0\prec x_1$ or $y_1\succ y_2\succ\ldots$), let $\Lyn\calX$ be the set of Lyndon words over $\calX$ (generating the free monoid $(\calX,1_{\calX^*})$)\footnote{Let $\mu$ denote the M\"obius function. Then the number of Lyndon words of length $k$, over $X$, is given by a Witt's formula \cite{lothaire}, \textit{i.e.}
\begin{eqnarray*}
\psi_2(k)=k^{-1}\sum_{d\vert k}\mu(d)2^{k/d}.
\end{eqnarray*}
Hence, with the notations in Conjecture \ref{Zconj}, one also has $d_k\le \psi_2(k)-\psi_2(k-1)$.}.
Using Notations in \eqref{piXY1}--\eqref{piXY2}, one has \cite{lothaire}
\begin{eqnarray}\label{perrin}
l\in\Lyn X\setminus\{x_0\}\iff\pi_Y(l)\in\Lyn Y.
\end{eqnarray}

Any pair of Lyndon words $(l_1,l_2)$ is called the standard factorization of $l\in\Lyn\calX$, and is denoted by $st(l)$, if $l=l_1l_2$ and $l_2$ is the longest nontrivial proper right factor of $l$ or, equivalently, its smallest such, for lexicographic ordering (see \cite{lothaire} for proofs). 

The $A$-module $\ncp{A}{\calX}$ is endowed with the unital associative concatenation product and the unital associative commutative shuffle product. The latter is recursively defined,
for any $x,y\in\calX$ and $u,v,w\in\calX^*$, by \cite{fliess1}
\begin{eqnarray}\label{shuffle}
w\shuffle 1_{\calX^*}=1_{\calX^*}\shuffle w=w&\mbox{and}&
xu\shuffle yv=x(u\shuffle yv)+y(xu\shuffle v).
\end{eqnarray}
The coproducts $\Delta_{\conc}$ and $\Delta_{\shuffle}$ are morphisms for concatenation and are defined by
\begin{eqnarray}\label{coproducts}
\forall x\in\calX,&\Delta_{\conc}(x)=\Delta_{\shuffle}(x)=1_{\calX^*}\otimes x+x\otimes1_{\calX^*}.
\end{eqnarray}

By a Radford's theorem \cite{radford}, $\Lyn\calX$ forms a pure transcendence basis of the shuffle algebra $(\ncp{A}{\calX},\shuffle,1_{\calX^*})$.
Letting $\{P_l\}_{l\in\Lyn\calX}$ be a homogeneous basis of the Lie algebra $\ncp{\calL ie_{A}}{\calX}$, the enveloping algebra $\mathcal{U}(\ncp{\calL ie_{A}}{\calX})$ is classically endowed with the graded linear basis $\{P_w\}_{w\in\calX^*}$ expanded, by PBW theorem \cite{Lie7,lothaire,reutenauer}, after $\{P_l\}_{l\in\Lyn\calX}$ and it is isomorphic, by CQMM theorem \cite{Cartier2}, to the following connected, graded and co-commutative bialgebra
\begin{eqnarray}\label{Hopfshuffle}
\calH_{\shuffle}(\calX):=(\ncp{A}{\calX},\conc,1_{\calX^*},\Delta_{\shuffle})
\end{eqnarray}

\begin{remark}\label{length}
For ${\calX}=X$ or $Y$ the corresponding monoids are equipped with length functions, for $X$ we consider the length of words (\textit{i.e.} $(w)=\ell(w)=\abs{w}$) and for $Y$ the length is given by the weight (\textit{i.e.} $(w)=\ell(y_{i_1}\ldots y_{i_n})=i_1+\ldots+i_n$). This naturally induces a grading of $\ncp{A}{{\calX}}$ and $\ncp{\calL ie_{A}}{{\calX}}$ in free modules of finite dimensions. Hence, with the notations used in \eqref{perrin}, for any $\lambda\in\Lyn X\setminus\{x_0\}$, one also has $\abs{\lambda}=(\pi_Y(\lambda))$.
\end{remark}

Moreover, using the following pairing
\begin{eqnarray}\label{pairing}
\ncs{A}{\calX}\otimes\ncp{A}{\calX}\longrightarrow A,
T\otimes P\longmapsto\scal{T}{P}:=\sum_{w\in\calX^*}\scal{T}{w}\scal{P}{w},
\end{eqnarray}
let the graded dual basis of $\{P_w\}_{w\in\calX^*}$ be $\{S_w\}_{w\in\calX^*}$, containing the transcendence basis $\{S_l\}_{l\in\Lyn\calX}$ of the shuffle algebra, and let the graded dual of $\calH_{\shuffle}(\calX)$ be
\begin{eqnarray}
\calH_{\shuffle}^{\vee}(\calX):=(\ncp{A}{\calX},{\shuffle},1_{\calX^*}\Delta_{\conc}).
\end{eqnarray}
Then one also has the following factorization of the diagonal series \cite{reutenauer}
\begin{eqnarray}\label{diagonalX}
{\calD}_{\calX}:=\sum_{w\in\calX^*}w\otimes w=\sum_{w\in\calX^*}S_w\otimes P_w
=\prod_{l\in\Lyn\calX}^{\searrow}e^{S_l\otimes P_l}.
\end{eqnarray}

The dual polynomial bases $\{P_w\}_{w\in\calX^*}$ and $\{S_w\}_{w\in\calX^*}$, homogeneous in weight (\textit{i.e.} length of each $w$), can be constructed recursively by (see \cite{Bui}, for examples by computer)
\begin{eqnarray}\label{P_w}
\left\{\begin{array}{rclrclll}
P_x&=&x,&S_x&=&x,
&\mbox{for }x\in\calX,\\
P_l&=&[P_{l_1},P_{l_2}],&S_l&=&yS_{l'},
&\mbox{for }l=yl'\in\Lyn\calX\setminus\calX,\atop st(l)=(l_1,l_2),\\
P_w&=&P_{l_1}^{i_1}\ldots P_{l_k}^{i_k},&S_w&=&\displaystyle\frac{S_{l_1}^{\shuffle i_1}\shuffle\ldots\shuffle S_{l_k}^{\shuffle i_k}}{i_1!\ldots i_k!},
&{\displaystyle\mbox{for }w=l_1^{i_1}\ldots l_k^{i_k},\mbox{ with }l_1,\ldots,\atop\displaystyle l_k\in\Lyn\calX,l_1\succ\ldots\succ l_k,}
\end{array}\right.
\end{eqnarray}

Let us recall also that the polynomials $\{S_l\}_{l\in\Lyn X}$ (resp. $\{P_l\}_{l\in\Lyn X}$), homogeneous in weight (\textit{i.e.} length of each $l$), are lower (resp. upper) triangular, \textit{i.e.} for any $l\in\Lyn X$ and $l'\in\Lyn Y$, one has (see \cite{reutenauer})
\begin{eqnarray}\label{triangular1}
S_l=l+\sum_{v\succ l,\abs{v}=\abs{l}}a_vv,\ a_v\in A&\mbox{and}&
P_l=l+\sum_{v\prec l,\abs{v}=\abs{l}}c_vv,\ c_v\in A.
\end{eqnarray}

\begin{remark}\label{ordered1}
\begin{enumerate}
\item The ordering over $\Lyn X$ elicits the one over $\{S_l\}_{l\in\Lyn X}$ and $\{P_l\}_{l\in\Lyn X}$.

\item By \eqref{triangular1}, the following assertions are equivalent
\begin{enumerate}
\item $\Lyn X$ is a pure transcendence basis of $(\ncp{A}{X},\shuffle,1_{X^*})$ 
(\textit{i.e.} the Radforf theorem \cite{radford,reutenauer}),
\item $\{S_l\}_{l\in\Lyn X}$  is a pure transcendence basis of $(\ncp{A}{X},\shuffle,1_{X^*})$.
\end{enumerate}

\item One also has
\begin{eqnarray*}
(A1_{X^*}\oplus x_0\ncp{A}{X}x_1,\shuffle,1_{X^*})\cong&A[\Lyn X\setminus X]&=A[\{S_l\}_{l\in\Lyn X\setminus X}].
\end{eqnarray*}
\end{enumerate}
\end{remark}

\begin{example}[\cite{hoangjacoboussous}]\label{P_lS_l}
For $X=\{x_0,x_1\}$, one has
$$\begin{array}{|c|c|c|}
\hline
l&P_l&S_l\\
\hline
x_0&x_0&x_0\\
x_1&x_1&x_1\\
x_0x_1&[x_0,x_1]&x_0x_1\\
x_0^2x_1&[x_0,[x_0,x_1]]&x_0^2x_1\\
x_0x_1^2&[[x_0,x_1],x_1]&x_0x_1^2\\
x_0^3x_1&[x_0,[x_0,[x_0,x_1]]]&x_0^3x_1\\
x_0^2x_1^2&[x_0,[[x_0,x_1],x_1]]&x_0^2x_1^2\\
x_0x_1^3&[[[x_0,x_1],x_1],x_1]&x_0x_1^3\\
x_0^4x_1&[x_0,[x_0,[x_0,[x_0,x_1]]]]&x_0^4x_1\\
x_0^3x_1^2&[x_0,[x_0,[[x_0,x_1],x_1]]]&x_0^3x_1^2\\
x_0^2x_1x_0x_1&[[x_0,[x_0,x_1]],[x_0,x_1]]&2x_0^3x_1^2+x_0^2x_1x_0x_1\\
x_0^2x_1^3&[x_0,[[[x_0,x_1],x_1],x_1]]&x_0^2x_1^3\\
x_0x_1x_0x_1^2&[[x_0,x_1],[[x_0,x_1],x_1]]&3x_0^2x_1^3+x_0x_1x_0x_1^2\\
x_0x_1^4&[[[[x_0,x_1],x_1],x_1],x_1]&x_0x_1^4\\
{x_0^5}x_1&[x_0,[x_0,[x_0,[x_0,[x_0,x_1]]]]]&x_0^5x_1\\\
x_0^4x_1^2&[x_0,[x_0,[x_0,[[x_0,x_1],x_1]]]]&x_0^4x_1^2\\
x_0^3x_1x_0x_1&[x_0,[[x_0,[x_0,x_1]],[x_0,x_1]]]&2x_0^4x_1^2+x_0^3x_1x_0x_1\\
x_0^3x_1^3&[x_0,[x_0,[[[x_0,x_1],x_1],x_1]]]&x_0^3x_1^3\\
x_0^2x_1x_0x_1^2&[x_0,[[x_0,x_1],[[x_0,x_1],x_1]]]&3x_0^3x_1^3+x_0^2x_1x_0x_1^2\\
x_0^2x_1^2x_0x_1&[[x_0,[[x_0,x_1],x_1]],[x_0,x_1]]&6x_0^3x_1^3+3x_0^2x_1x_0x_1^2+x_0^2x_1^2x_0x_1\\
x_0^2x_1^4&[x_0,[[[[x_0,x_1],x_1],x_1],x_1]]&x_0^2x_1^4\\
x_0x_1x_0x_1^3&[[x_0,x_1],[[[x_0,x_1],x_1],x_1]]&4x_0^2x_1^4+x_0x_1x_0x_1^3\\
x_0x_1^5&[[[[[x_0,x_1],x_1],x_1],x_1],x_1]&x_0x_1^5\\
\hline
\end{array}$$
\end{example}

\begin{remark}\label{weightX}
Using the notations in \eqref{piXY1}--\eqref{perrin} and extending Remark \ref{length}, since
\begin{eqnarray*}
\forall \lambda=x_{i_1}\cdots x_{i_r}\in\Lyn X\setminus\{x_0\},&
\abs{\lambda}=r
\end{eqnarray*}
and, by linear extension and by homogeneity in weight, that is length, of the polynomials $\{S_l\}_{l\in\Lyn X}$ and $\{P_l\}_{l\in\Lyn X}$, one has
\begin{eqnarray*}
\abs{S_{\lambda}}=r&\mbox{and}&\abs{P_{\lambda}}=r
\end{eqnarray*}
then the depth of $\zeta(S_{\lambda})$ is the depth of $\zeta(\lambda)$, that is $r$, as previously introduced in Section \ref{introduction}.
\end{remark}

In addition, $\ncp{A}{Y}$ is equipped with the unital associative commutative quasi-shuffle product recursively defined, for any $u,v,w\in Y^*$ and $y_i,y_j\in Y$, by \cite{SLC43,FPSAC97}
\begin{eqnarray}\label{stuffle}
w\stuffle 1_{Y^*}=1_{Y^*}\stuffle w=w&\mbox{and}&
y_iu\stuffle y_jv=y_i(u\stuffle y_jv)+y_j(y_iu\stuffle v)+y_{i+j}(u\stuffle v)
\end{eqnarray}
and its dual law, being a $\conc$-morphism, is given by
\begin{eqnarray}\label{Dstuffle}
\forall y_k\in Y,&\Delta_{\stuffle}(y_k)=y_k\otimes 1_{Y^*}+1_{Y^*}\otimes y_k+\displaystyle\sum_{i+j=k}y_i\otimes y_j.
\end{eqnarray}
Letting $\mathrm{Prim}(\calH_{\stuffle}(Y))=\span_{A}\{\pi_1(w)\}_{w\in Y^*}$ and $\pi_1$ be the eulerian projector defined\footnote{By \eqref{coproducts}, any letter $x\in\calX$ is primitive, for $\Delta_{\conc}$ and $\Delta_{\shuffle}$. By \eqref{Dstuffle}, the polynomials $\{\pi_1(y_k)\}_{k\ge2}$ and only the letter $y_1$ are primitive, for $\Delta_{\stuffle}$.}, for any $w\in Y^*$, by \cite{CM}
\begin{eqnarray}\label{piii_1}
\pi_1(w)=w+\sum_{k=2}^{(w)}\frac{(-1)^{k-1}}k\sum_{u_1,\ldots,u_k\in Y^+}\scal{w}{u_1\stuffle\ldots\stuffle u_k}u_1\ldots u_k.
\end{eqnarray}
these yield another connected, graded and co-commutative bialgebra being isomorphic to the enveloping algebra of the Lie algebra of primitive elements, \textit{i.e.}
\begin{eqnarray}\label{Hopfstuffle}
\calH_{\stuffle}(Y):=(\ncp{A}{Y},\conc,1_{Y^*},\Delta_{\stuffle})\cong\mathcal{U}(\mathrm{Prim}(\calH_{\stuffle}(Y))).
\end{eqnarray}

Let $\{\Pi_w\}_{w\in Y^*}$ be a linear basis, homogeneous in weight (\textit{i.e.} weight of each $w$), expanded by decreasing PBW after a basis $\{\Pi_l\}_{l\in \Lyn Y}$ of $\mathrm{Prim}(\calH_{\stuffle}(Y))$ and let $\{\Sigma_w\}_{w\in Y^*}$ (homogeneous in weight) be its dual basis containing the transcendence basis $\{\Sigma_l\}_{l\in\Lyn Y}$ of the quasi-shuffle algebra. The graded dual of $\calH_{\stuffle}(Y)$ is denoted as follows
\begin{eqnarray}
\calH_{\stuffle}^{\vee}(Y)=(\ncp{A}{Y},{\stuffle},1_{Y^*},\Delta_{\conc}).
\end{eqnarray}
The diagonal series ${\calD}_Y$, on $\calH_{\stuffle}(Y)$ is also factorized as follows \cite{CM} 
\begin{eqnarray}\label{diagonalY}
{\calD}_Y:=\sum_{w\in Y^*}w\otimes w=
\sum_{w\in Y^*}\Sigma_w\otimes\Pi_w
=\prod_{l\in\Lyn Y}^{\searrow}e^{\Sigma_l\otimes\Pi_l}.
\end{eqnarray}

\begin{theorem}[\cite{CM}]\label{isomorphy}
Let $\varphi_{\pi_1}:(\ncp{A}{Y},\conc,1_{Y^*})\longrightarrow(\ncp{A}{Y},\conc,1_{Y^*})$ maps
$y_k$ to $\pi_1(y_k)$. It is an automorphism of $\ncp{A}{Y}$ realizing an isomorphism of bialgebras
between $\calH_{\shuffle}(Y)$ and $\calH_{\stuffle}(Y)$. The following diagram commutes
\begin{center}
\begin{tikzcd}[column sep=3em]
\ncp{A}{Y}\ar[hook]{r}{\Delta_{\shuffle}}\ar[swap]{d}{\varphi_{\pi_1}}& 
\ncp{A}{Y}\otimes\ncp{A}{Y}\ar[]{d}{\varphi_{\pi_1}\otimes\varphi_{\pi_1}}\\
\ncp{A}{Y}
\ar[hook]{r}{\Delta_{\stuffle}}& 
\ncp{A}{Y}\otimes\ncp{A}{Y} 
\end{tikzcd}
\end{center}
and $\{\Pi_w\}_{w\in Y^*}$ (resp. $\{\Sigma_w\}_{w\in Y^*}$) is image of $\{P_w\}_{w\in Y^*}$ (resp. $\{S_w\}_{w\in Y^*}$) by
$\varphi_{\pi_1}$ (resp. $\check\varphi_{\pi_1}^{-1}$), where $\check\varphi_{\pi_1}$ is the adjoint of $\varphi_{\pi_1}$.
\end{theorem}

Algorithmically, the bases of homogeneous in weight polynomials $\{P_w\}_{w\in Y^*}$ and
$\{S_w\}_{w\in Y^*}$ of, respectively, $\calU(\mathrm{Prim}(\calH_{\stuffle}(Y)))$
and $(\ncp{A}{Y},\shuffle,1_{Y^*})$ can be directly and recursively constructed as follows
(see \cite{Bui}, for examples by computer)
\begin{eqnarray}\label{Pi_w}
\left\{\begin{array}{rclrclll}
\Pi_{y_s}&=&\pi_1(y_s),&\Sigma_{y_s}&=&y_s
&\mbox{for }y_s\in Y,\\
\Pi_{l}&=&[\Pi_{l_1},\Pi_{l_2}],&\Sigma_l&=&\sum_{(*)}\frac{y_{s_{k_1}+\ldots+s_{k_i}}}{i!}\Sigma_{l_1\ldots l_n},
&\mbox{for }l\in\Lyn Y\setminus Y,\atop st(l)=(l_1,l_2),\\
\Pi_{w}&=&\Pi_{l_1}^{i_1}\ldots\Pi_{l_k}^{i_k},&\Sigma_w&=&\displaystyle\frac{\Sigma_{l_1}^{\stuffle i_1}\stuffle\ldots\stuffle\Sigma_{l_k}^{\stuffle i_k}}{i_1!\ldots i_k!},
&{\displaystyle\mbox{for }w=l_1^{i_1}\ldots l_k^{i_k},\mbox{ with }l_1,\ldots,\atop\displaystyle l_k\in\Lyn Y,l_1\succ\ldots\succ l_k,}
\end{array}\right.
\end{eqnarray}
In $(*)$, the sum is taken over all $\{k_1,\ldots,k_i\}\subset\{1,\ldots,k\}$
and $l_1\succeq\ldots\succeq l_n$ such that $(y_{s_1},\ldots,y_{s_k})$ is derived from $(y_{s_{k_1}},\ldots,y_{s_{k_i}},l_1,\ldots,l_n)$ by transitive closure of the relations on standard sequences \cite{SLC74,reutenauer}.

\begin{remark}\label{ordered2}
\begin{enumerate}
\item For any $n\ge1$, one has
\begin{eqnarray*}
(A1_{Y^*}\oplus(Y\setminus\{y_1\})\ncp{A}{Y},\stuffle,1_{X^*})\cong&
A[\Lyn Y\setminus\{y_1\}]&=A[\{\Sigma_l\}_{l\in\Lyn Y\setminus\{y_1\}}].
\end{eqnarray*}

\item The ordering over $\Lyn Y$ elicits the ordering over$\{\Sigma_l\}_{l\in\Lyn Y}$ and $\{\Pi_l\}_{l\in\Lyn Y}$.

\item By \eqref{triangular2}, the following assertions are equivalent
\begin{enumerate}
\item $\Lyn Y$ is a pure transcendence basis of $(\ncp{A}{Y},\stuffle,1_{Y^*})$ (\textit{i.e.} an extended Radforf theorem \cite{acta,VJM,hoffman2}),
\item $\{\Sigma_l\}_{l\in\Lyn Y}$)is a pure transcendence basis of $(\ncp{A}{Y},\stuffle,1_{Y^*})$.
\end{enumerate}
\end{enumerate}
\end{remark}

\begin{example}[\cite{VJM}]\label{Pi_lSigma_l}
For $Y=\{y_k\}_{k\ge1}$, one has
$$\begin{array}{|c|c|c|}
\hline
w&\Pi_w&\Sigma_w\\
\hline
y_2&y_2-\frac{1}{2}y_1^2&y_2\\
y_1^2&y_1^2&\frac{1}{2}y_2+y_1^2\\
y_3&y_3-\frac{1}{2}y_1y_2-\frac{1}{2}y_2y_1+\frac{1}{3}y_1^3&y_3\\
y_2y_1&y_2y_1-y_2y_1&\frac{1}{2}y_3+y_2y_1\\
y_1y_2&y_2y_1-\frac{1}{2}y_1^3&y_2y_1+\frac{1}{2}y_3\\
y_1^3&y_1^3&\frac{1}{6}y_3+\frac{1}{2}y_2y_1+\frac{1}{2}y_1y_2+y_1^3\\
y_4&y_4-\frac{1}{2}y_1y_3-\frac{1}{2}y_2^2-\frac{1}{2}y_3y_1&y_4\\
&+\frac{1}{3}y_1^2y_2+\frac{1}{3}y_1y_2y_1+\frac{1}{3}y_2y_1^2-\frac{1}{4}y_1^4&\\
y_3y_1&y_3y_1-\frac{1}{2}y_2y_1^2-y_1y_3+\frac{1}{2}y_1^2y_2&\frac{1}{2}y_4+y_3y_1\\
y_2^2&y_2^2-\frac{1}{2}y_2y_1^2-\frac{1}{2}y_1^2y_2+\frac{1}{4}y_1^4&\frac{1}{2}y_4+y_2^2\\
y_2y_1^2&y_2y_1^2-2\,y_1y_2y_1+y_1^2y_2&\frac{1}{6}y_4+\frac{1}{2}y_3y_1+\frac{1}{2}y_2^2+y_2y_1^2\\
y_1y_3&y_1y_3-\frac{1}{2}y_1^2y_2-\frac{1}{2}y_1y_2y_1+\frac{1}{3}y_1^4&y_4+y_3y_1+y_1y_3\\
y_1y_2y_1&y_1y_2y_1-y_1^2y_2&\frac{1}{2}y_4+\frac{1}{2}y_3y_1+y_2^2\\
&&+y_2y_1^2+\frac{1}{2}y_1y_3+y_1y_2y_1\\
y_1^2y_2&y_1^2y_2-\frac{1}{2}y_1^4&\frac{1}{2}y_4+y_3y_1+y_2^2+y_2y_1^2\\
&&+y_1y_3+y_1y_2y_1+y_1^2y_2\\
y_1^4&y_1^4&\frac{1}{24}y_4+\frac{1}{6}y_3y_1+\frac{1}{4}y_2^2+\frac{1}{2}y_2y_1^2\\
&&+\frac{1}{6}y_1y_3+\frac{1}{2}y_1y_2y_1+\frac{1}{2}y_1^2y_2+y_1^4\\
\hline
\end{array}$$
\end{example}

The homogeneous in weight polynomials $\{\Sigma_l\}_{l\in\Lyn Y}$ (resp. $\{\Pi_l\}_{l\in\Lyn Y}$) are also lower (resp. upper) triangular, \textit{i.e.} for any $l\in\Lyn X$ and $l'\in\Lyn Y$, one has \cite{acta,VJM}
\begin{eqnarray}\label{triangular2}
\Sigma_{l'}=l'+\sum_{v\succ l',(v)=(l')}b_vv,\ b_v\in A&\mbox{and}&
\Pi_{l'}=l'+\sum_{v\prec l',(v)=(l')}d_vv,\ d_v\in A.
\end{eqnarray}

\begin{remark}\label{weightY}
Extending Remark \ref{length} and using the notations in \eqref{piXY1}--\eqref{perrin}, since
\begin{eqnarray*}
\forall \lambda=y_{i_1}\cdots y_{i_r}\in\Lyn Y,&(\lambda)=\abs{\pi_X(\lambda)}=i_1+\cdots+i_r
\end{eqnarray*}
and, by linear extension and by homogeneity (in weight) of the polynomials $\{\Sigma_l\}_{l\in\Lyn Y}$ and $\{\Pi_l\}_{l\in\Lyn Y}$, one has
\begin{eqnarray*}
(\Sigma_{\lambda})=\abs{\pi_X(\lambda)}=i_1+\cdots+i_r,
&\mbox{and}&(\Pi_{\lambda})=\abs{\pi_X(\lambda)}=i_1+\cdots+i_r
\end{eqnarray*}
then the weight of $\zeta(\Sigma_{\lambda})$ is the weight of $\zeta(\lambda)$, that is $i_1+\cdots+i_r$, as previously introduced in Section \ref{introduction}.
\end{remark}

\section{Polylogarithms, harmonic sums and polyzetas viewed as functions on free monoids and their graphs}\label{Indexation}
\begin{theorem}[\cite{CM}]\label{structure1}
With the notations given in \eqref{def1def2} and \eqref{indexation}, the following morphisms of algebras are \textit{bijective}
\begin{eqnarray*}
\H_{\bullet}:(\Q\pol{Y},\stuffle,1_{Y^*})&\hooklongtwoheadrightarrow&
(\mathrm{span}_{\Q}\{\H_w\}_{w\in Y^*},\times,1),\\
y_{s_1}\cdots y_{s_k}&\longmapsto&\H_{s_1,\ldots,s_r},\\
\Li_{\bullet}:(\QX,\shuffle,1_{X^*})&\hooklongtwoheadrightarrow&
(\mathrm{span}_{\Q}\{\Li_w\}_{w\in X^*},\times,1_{\widetilde{\C\setminus\{0,1\}}}),\\
x_0&\longmapsto&\log z,\\
x_0x_1^{s_1-1}\ldots x_0x_1^{s_k-1}&\longmapsto&\Li_{s_1,\ldots,s_r}.
\end{eqnarray*}
Or equivalently, $\{\H_w\}_{w\in Y^*}$ (resp. $\{\H_{\Sigma_l}\}_{l\in\Lyn Y}$ and $\{\H_l\}_{l\in\Lyn Y}$) and $\{\Li_w\}_{w\in X^*}$
(resp. $\{\Li_{S_l}\}_{l\in\Lyn X}$ and $\{\Li_l\}_{l\in\Lyn X}$) are linearly (resp. algebraically) independent over $\Q$.
\end{theorem}

\begin{corollary}[\cite{CM}]\label{zeta}
Hence, with the notations given in \eqref{Image} and \eqref{convenience}, the following polymorphism is \textit{surjective}
\begin{eqnarray*}
\zeta:\left\{{(\Q1_{Y^*}\oplus(Y\setminus\{y_1\})\Q\pol{Y},\stuffle,1_{Y^*})\atop
(\Q1_{X^*}\oplus x_0\Q\pol{X}x_1,\shuffle,1_{X^*})}\right\}
&\longtwoheadrightarrow&(\calZ,\times,1),\\
\left\{{y_{s_1}\ldots y_{s_k}\atop x_0x_1^{s_1-1}\ldots x_0x_1^{s_k-1}}\right\}
&\longmapsto&\zeta(s_1,\ldots,s_r),
\end{eqnarray*}

Moreover, the $\Q$-algebra $(\calZ,\times,1)$ is a subalgebra of $(\R,\times,1)$ and is generated by $\{\zeta(l)\}_{l\in\Lyn\calX\setminus\gDIV}$, or equivalently by $\{\zeta(\Sigma_l)\}_{l\in\Lyn Y\setminus\{y_1\}}$ or by $\{\zeta(S_l)\}_{l\in\Lyn X\setminus X}$. 
\end{corollary}

Using the notations in Conjecture \ref{Zconj}, it follows also that
\begin{eqnarray}\label{prod}
\forall k,k'\ge1,&\calZ_k\calZ_{k'}\subset\calZ_{k+k'}.
\end{eqnarray}

\begin{remark}
The polyzetas $\{\zeta(l)\}_{l\in\Lyn\calX\setminus\gDIV}$ are not necessary algebraically independent over $\Q$ (see the Euler's relations in \eqref{zeta(s,1)}), contrary to $\{\H_l\}_{l\in\Lyn Y}$ and $\{\Li_l\}_{l\in\Lyn X}$  (resp. $\{\H_{\Sigma_l}\}_{l\in\Lyn Y}$ and $\{\Li_{S_l}\}_{l\in\Lyn X}$). Hence, it is legitimate to determine the $\Q$-algebraic relations among them. So does \cite{FPSAC97}, in which it is conjectured that the $\Q$-algebra $(\calZ,\times,1)$ is $\Q$-free and then the objective of this present work is to construct effectively an algebraic basis for $(\calZ,\times,1)$ by expliciting the \textit{non} $\Q$-algebraic relations among $\{\zeta(\Sigma_l)\}_{l\in\Lyn Y\setminus\{y_1\}}$ and $\{\zeta(S_l)\}_{l\in\Lyn X\setminus X}$) (see Proposition \ref{directsum} below).
\end{remark}

\begin{example}\label{Pi_lSigma_lsuite}
With the expressions of $\{\Sigma_l\}_{l\in\Lyn Y}$ in Example \ref{Pi_lSigma_l}, using the $\zeta$ polymorphism and \eqref{indexation}, one has
\begin{eqnarray*}
\zeta(\Sigma_{y_2})&=&\zeta(y_2)\\
&=&\zeta(2),\cr
\zeta(\Sigma_{y_3})&=&\zeta(y_3)\\
&=&\zeta(3),\cr
\zeta(\Sigma_{y_2y_1})&=&\frac{1}{2}\zeta(y_3)+\zeta(y_2y_1)\\
&=&\frac{1}{2}\zeta(3)+\zeta(2,1),\cr
\zeta(\Sigma_{y_4})&=&\zeta(y_4)\\
&=&\zeta(4),\cr
\zeta(\Sigma_{y_3y_1})&=&\frac{1}{2}\zeta(y_4)+\zeta(y_3y_1)\\
&=&\frac{1}{2}\zeta(4)+\zeta(3,1),\cr
\zeta(\Sigma_{y_2^2})&=&\frac{1}{2}\zeta(y_4)+\zeta(y_2y_2)\\
&=&\frac{1}{2}\zeta(4)+\zeta(2,2),\cr
\zeta(\Sigma_{y_2y_1^2})&=&\frac{1}{6}\zeta(y_4)+\frac{1}{2}\zeta(y_3y_1)+\frac{1}{2}\zeta(y_2^2)+\zeta(y_2y_1^2)\\
&=&\frac{1}{6}\zeta(4)+\frac{1}{2}\zeta(3,1)+\frac{1}{2}\zeta(2,2)+\zeta(2,1,1),\cr
\zeta(\Sigma_{y_1y_3})&=&\zeta(y_4)+\zeta(y_3y_1)+\zeta(y_1y_3)\\
&=&\zeta(4)+\zeta(3,1)+\zeta(1,3),\cr
\zeta(\Sigma_{y_1y_2y_1})&=&\frac{1}{2}\zeta(y_4)+\frac{1}{2}\zeta(y_3y_1)+\zeta(y_2^2)+\zeta(y_2y_1^2)+\frac{1}{2}\zeta(y_1y_3)+\zeta(y_1y_2y_1)\\
&=&\frac{1}{2}\zeta(4)+\frac{1}{2}\zeta(3,1)+\zeta(2,2)+\zeta(2,1,1)+\frac{1}{2}\zeta(1,3)+\zeta(1,2,1).
\end{eqnarray*}
\end{example}

\begin{example}\label{P_lS_lsuite}
With the expressions of $\{S_l\}_{l\in\Lyn X}$ in Example \ref{P_lS_l}, using the $\zeta$ polymorphism and \eqref{indexation}, one has
\begin{eqnarray*}
\zeta(S_{x_0x_1})&=&\zeta(x_0x_1)\\&=&\zeta(2),\cr
\zeta(S_{x_0^2x_1})&=&\zeta(x_0^2x_1)\\&=&\zeta(3),\cr
\zeta(S_{x_0x_1^2})&=&\zeta(x_0x_1^2)\\&=&\zeta(2,1),\cr
\zeta(S_{x_0^3x_1})&=&\zeta(x_0^3x_1)\\&=&\zeta(4),\cr
\zeta(S_{x_0^2x_1^2})&=&\zeta(x_0^2x_1^2\\&=&\zeta(3,1),\cr
\zeta(S_{x_0x_1^3})&=&\zeta(x_0x_1^3)\\&=&\zeta(2,1,1),\cr
\zeta(S_{x_0^4x_1})&=&\zeta(x_0^4x_1)\\&=&\zeta(5),\cr
\zeta(S_{x_0^3x_1^2})&=&\zeta(x_0^3x_1^2)\\&=&\zeta(4,1),\cr
\zeta(S_{x_0^2x_1x_0x_1})&=&2\zeta(x_0^3x_1^2)+\zeta(x_0^2x_1x_0x_1)\\&=&2\zeta(4,1)+\zeta(3,2),\cr
\zeta(S_{x_0^2x_1^3})&=&\zeta(x_0^2x_1^3)\\&=&\zeta(3,1,1),\cr
\zeta(S_{x_0x_1x_0x_1^2})&=&3\zeta(x_0^2x_1^3)+\zeta(x_0x_1x_0x_1^2)\\&=&3\zeta(3,1,1)+\zeta(2,2,1),\cr
\zeta(S_{x_0x_1^4})&=&\zeta(x_0x_1^4)\\&=&\zeta(2,1,1,1),\cr
\zeta(S_{{x_0^5}x_1})&=&\zeta(x_0^5x_1)\\&=&\zeta(6),\cr
\zeta(S_{x_0^4x_1^2})&=&\zeta(x_0^4x_1^2)\\&=&\zeta(5,1),\cr
\zeta(S_{x_0^3x_1x_0x_1})&=&2\zeta(x_0^4x_1^2)+\zeta(x_0^3x_1x_0x_1)\\&=&2\zeta(5,1)+\zeta(4,2),\cr
\zeta(S_{x_0^3x_1^3})&=&\zeta(x_0^3x_1^3)\\&=&\zeta(4,1,1),\cr
\zeta(S_{x_0^2x_1x_0x_1^2})&=&3\zeta(x_0^3x_1^3)+\zeta(x_0^2x_1x_0x_1^2)\\&=&3\zeta(4,1,1)+\zeta(3,2,1),\cr
\zeta(S_{x_0^2x_1^2x_0x_1})&=&6\zeta(x_0^3x_1^3)+3\zeta(x_0^2x_1x_0x_1^2)+\zeta(x_0^2x_1^2x_0x_1)\\&=&6\zeta(4,1,1)+3\zeta(3,2,1)+\zeta(3,1,2),\cr
\zeta(S_{x_0^2x_1^4})&=&\zeta(x_0^2x_1^4)\\&=&\zeta(2,1,1,1),\cr
\zeta(S_{x_0x_1x_0x_1^3})&=&4\zeta(x_0^2x_1^4)+\zeta(x_0x_1x_0x_1^3)\\&=&4\zeta(3,1,1,1)+\zeta(2,2,1,1),\cr
\zeta(S_{x_0x_1^5})&=&\zeta(x_0x_1^5)\\&=&\zeta(2,1,1,1,1).
\end{eqnarray*}
\end{example}

Now, we are in situation to consider the graphs of $\Li_{\bullet}$ \cite{FPSAC98} and of $\H_{\bullet}$ \cite{Daresbury,JSC}, using the diagonal series ${\mathcal D}_X$ and ${\mathcal D}_Y$, as follows
\begin{eqnarray}\label{NCGS}
\L:=\sum_{w\in X}\Li_ww=(\Li_{\bullet}\otimes\mathrm{Id}){\mathcal D}_X
&\mbox{and}&
\H:=\sum_{w\in Y}\H_ww=(\H_{\bullet}\otimes\mathrm{Id}){\mathcal D}_Y.
\end{eqnarray}
Then by \eqref{diagonalX} and \eqref{diagonalY} these graphs are grouplike series respectively, for $\shuffle$ and $\stuffle$, and are put in the MRS forms, leading to the definitions of grouplike series $Z_{\shuffle}$ \cite{FPSAC98} and $Z_{\stuffle}$ \cite{acta,VJM,CM} respectively, for $\shuffle$ and $\stuffle$, by setting
\begin{eqnarray}
\L=\prod_{l\in\Lyn X}^{\searrow}e^{\Li_{S_l}P_l}
&\mbox{and let}&
Z_{\shuffle}:=\prod_{l\in\Lyn X\setminus X}^{\searrow}e^{\zeta(S_l)P_l},\label{Zshuffle}\\
\H=\prod_{l\in\Lyn Y}^{\searrow}e^{\H_{\Sigma_l}\Pi_l}
&\mbox{and let}&
Z_{\stuffle}:=\prod_{l\in\Lyn Y\setminus\{y_1\}}^{\searrow}e^{\zeta(\Sigma_l)\Pi_l}.\label{Zstuffle}
\end{eqnarray}

Let also $Z_{\gamma}$ be the following noncommutative series
\begin{eqnarray}\label{Zgamma}
Z_{\gamma}:=\sum_{w\in Y^*}\gamma_ww,
\end{eqnarray}
where\footnote{$\gamma_{y_1}$ is nothing other than the Euler constant, usual denoted by $\gamma$.} $\gamma_w$ denotes the finite part ($\mathrm{f.p.}$ for shorted)\footnote{The notations in \eqref{gamma_bullet} and \eqref{regularization1}-\eqref{regularization2} are volontary used to justify and to clarify various regularizations discussed in Remark \ref{regularizations} bellow.} of the asymptotic expansion of the hamonic sum $\H_w$ in the comparison scale $\{n^a\log^b(n)\}_{a\in\Z,b\in\N}$, for $n\to+\infty$:
\begin{eqnarray}\label{gamma_bullet}
\scal{Z_{\gamma}}{w}=\gamma_w:=\mathop{\mathrm{f.p.}}_{n\rightarrow+\infty}\H_w(n),\{n^a\log^b(n)\}_{a\in\Z\atop b\in\N}.
\end{eqnarray}
It follows that $Z_{\gamma}$ is the graph of the following $\stuffle$-character \cite{Daresbury,JSC}
\begin{eqnarray}\label{gamma_character}
\gamma_{\bullet}:(\Q\pol{Y},\stuffle,1_{Y^*})&\longrightarrow&({\mathcal Z}[\gamma],\times,1),\\
y_1&\longmapsto&\gamma,\\
\forall l\in\Lyn Y\setminus\{y_1\}&\longmapsto&\gamma_{l}=\zeta(l),
\end{eqnarray}
for which, for any $l_1$ and $l_2\in\Lyn Y$, one also has \cite{Daresbury,JSC}
\begin{eqnarray}
\gamma_{l_1\stuffle l_2}=\gamma_{l_1}\gamma_{l_2}&\mbox{and}&
\gamma_{\Sigma_{l_1}\stuffle\Sigma_{l_2}}=\gamma_{\Sigma_{l_1}}\gamma_{\Sigma_{l_2}}.
\end{eqnarray}
Therefore, $Z_{\gamma}$ is grouplike, for $\Delta_{\stuffle}$, and one has \cite{acta,VJM}
\begin{eqnarray}
Z_{\gamma}=e^{\gamma y_1}Z_{\stuffle}.
\end{eqnarray}

On the other hand, $Z_{\stuffle}$ and $Z_{\shuffle}$ are also viewed as graphs of the following characters, extending the $\zeta$ polymorphism, \textit{i.e.}\footnote{With the comparison scale $\{(1-z)^a\log^b((1-z)^{-1})\}_{a\in\Z,b\in\N}$ (resp. $\{n^a\H_1^b(n)\}_{a\in\Z,b\in\N}$), the finite part of the singular (resp. asymptotic) expansion of $\Li_u$ (resp. $\H_v$) equals $\zeta_{\shuffle}(u)$ (resp. $\zeta_{\stuffle}(v)$), while with the comparison scale $\{n^a\log^b(n)\}_{a\in\Z,b\in\N}$, the finite part of the asymptotic expansion of $\H_w$ equals $\gamma_w$ (see \eqref{Zgamma} above).}
\begin{eqnarray}\label{regularization0}
\zeta_{\stuffle}:(\Q\pol{Y},\stuffle,1_{Y^*})\longrightarrow({\mathcal Z},\times,1),&
\zeta_{\shuffle}:(\QX,\shuffle,1_{X^*})\longrightarrow({\mathcal Z},\times,1),
\end{eqnarray}
and, for any word $w$ in $Y^*$ (resp $X^*$), each polyzeta $\zeta_{\stuffle}(w)$ (resp. $\zeta_{\shuffle}(w)$) represents the finite part of the asymptotic expansion (resp. functional expansion) of the hamonic sum $\H_w$ (resp. polylogarithm $\Li_w$) in the comparison scale $\{n^a\H_1^b(n)\}_{a\in\Z,b\in\N}$ (resp. $\{(1-z)^a\log^b((1-z)^{-1})\}_{a\in\Z,b\in\N}$), for $n\to+\infty$ (resp. $z\to1$):
\begin{eqnarray}
\zeta_{\stuffle}(w)=\scal{Z_{\stuffle}}{w}=
&\mathop{\mathrm{f.p.}}\limits_{n\rightarrow+\infty}\H_w(n),
&\{n^a\H_1^b(n)\}_{a\in\Z\atop b\in\N},\label{regularization1}\\
\zeta_{\shuffle}(w)=\scal{Z_{\shuffle}}{w}=
&\mathop{\mathrm{f.p.}}\limits_{z\rightarrow1}\Li_w(z),
&\{(1-z)^a\log^b((1-z)^{-1})\}_{a\in\Z\atop b\in\N},.\label{regularization2}
\end{eqnarray}
Hence, in virtue of \eqref{zetavalues}, \eqref{perrin} and \eqref{gamma_character}, the following convergent polyzetas equal:
\begin{eqnarray}\label{reg}
\forall l\in\Lyn Y\setminus\{y_1\},&
\zeta_{\stuffle}(l)=\zeta_{\shuffle}(\pi_X(l))=\zeta(l)=\gamma_l
\end{eqnarray}
and, on the other hand, one has for the divergent generators of weight $1$,
\begin{eqnarray}
\zeta_{\shuffle}(x_0)=\Li_{x_0}(1)=\log(1)=0
\end{eqnarray}
and, as consequence of the finte parts in \eqref{regularization1}--\eqref{regularization2}\footnote{Studying the commutative generating series of the double polyzetas $\{\zeta(s_1,s_2)\}_{s_1>1,s_2\ge1}$, Cartier used the convention $\zeta(1)=0$, in \cite{Cartier_double}.},
\begin{eqnarray}\label{simul_regu}
\zeta_{\shuffle}(x_1)=\zeta_{\stuffle}(y_1)=0.
\end{eqnarray}

\begin{proposition}\label{noyeaux}
For any $l_1$ and $l_2\in\Lyn X\setminus\{x_0\}$ (resp. $\Lyn Y$), by \eqref{piXY1}--\eqref{piXY2}, one has\footnote{\label{different} For any $l_1$ and $l_2\in\Lyn X\setminus\{x_0\}$ (resp. $\Lyn Y$), one also respectively has\\
$\begin{array}{c}
\zeta_{\shuffle}(S_{l_1}\shuffle S_{l_2})
=\zeta_{\shuffle}(S_{l_1})\zeta_{\shuffle}(S_{l_2})
=\zeta_{\stuffle}(\pi_Y(S_{l_1}))\zeta_{\stuffle}(\pi_Y(S_{l_2}))
=\zeta_{\stuffle}(\pi_Y(S_{l_1})\stuffle\pi_Y(S_{l_2})),\\
\zeta_{\stuffle}(\Sigma_{l_1}\stuffle\Sigma_{l_2})
=\zeta_{\stuffle}(\Sigma_{l_1})\zeta_{\stuffle}(\Sigma_{l_2})
=\zeta_{\shuffle}(\pi_X(\Sigma_{l_1}))\zeta_{\shuffle}(\pi_X(\Sigma_{l_2}))
=\zeta_{\shuffle}(\pi_X(\Sigma_{l_1})\shuffle\pi_X(\Sigma_{l_2})).
\end{array}$

But, for any $l\in\Lyn X\setminus\{x_0\}$ (resp. $\Lyn Y$), in general and unfortunately, $\pi_Y(S_l)\neq\Sigma_l$ (resp. $\pi_X(\Sigma_l)\neq S_l$) and then $\zeta_{\stuffle}(\pi_Y(S_l))\neq\zeta_{\stuffle}(\Sigma_l)$ (resp. $\zeta_{\shuffle}(\pi_X(\Sigma_l))\neq\zeta_{\shuffle}(S_l)$).}
\begin{eqnarray*}
\zeta_{\shuffle}(l_1\shuffle l_2)=\zeta_{\shuffle}(l_1)\zeta_{\shuffle}(l_2)
=\zeta_{\stuffle}(\pi_Y(l_1))\zeta_{\stuffle}(\pi_Y(l_2))=\zeta_{\stuffle}(\pi_Y(l_1)\stuffle\pi_Y(l_2)),\\
\zeta_{\stuffle}(l_1\stuffle l_2)=\zeta_{\stuffle}(l_1)\zeta_{\stuffle}(l_2)
=\zeta_{\shuffle}(\pi_X(l_1))\zeta_{\shuffle}(\pi_X(l_2))=
\zeta_{\shuffle}(\pi_X(l_1)\shuffle\pi_X(l_2)).
\end{eqnarray*}
\end{proposition}

\begin{proof}
As graphs of the morphisms $\zeta_{\shuffle}$ and $\zeta_{\stuffle}$, the series $Z_{\shuffle}$ and $Z_{\stuffle}$ are grouplike, for $\Delta_{\shuffle}$ and $\Delta_{\stuffle}$, respectively. Then, by the Friedrichs criterions for $\shuffle$ \cite{reutenauer} and $\stuffle$ \cite{VJM} and \eqref{regularization0}--\eqref{regularization2}, for any $u$ and $v\in x_0\ncp{\Q}{X}x_1$ (resp. $(Y\setminus\{y_1\})\ncp{\Q}{Y}$), one repectively has
\begin{eqnarray*}
\zeta_{\shuffle}(u\shuffle v)=\zeta_{\shuffle}(u)\zeta_{\shuffle}(v)
=\zeta_{\stuffle}(\pi_Y(u))\zeta_{\stuffle}(\pi_Y(v))=\zeta_{\stuffle}(\pi_Y(u)\stuffle\pi_Y(v)),\\
\zeta_{\stuffle}(u\stuffle v)=\zeta_{\stuffle}(u)\zeta_{\stuffle}(v)
=\zeta_{\shuffle}(\pi_X(u))\zeta_{\shuffle}(\pi_X(v))=
\zeta_{\shuffle}(\pi_X(u)\shuffle\pi_X(v)).
\end{eqnarray*}
In particular, using the $\Q$-algebraic basis, $\Lyn X\setminus\{x_0\}$ (resp. $\Lyn Y$) of $(\Q1_{X^*}\oplus\ncp{\Q}{X}x_1,\shuffle,1_{X^*})$ (resp. $(\ncp{\Q}{Y},\stuffle,1_{Y^*})$), these are nothing else the expected results.
\end{proof}

\begin{remark}\label{regularizations}
\begin{enumerate}
\item It was conjectured that $\{l_1\shuffle l_2-\pi_X(\pi_Y(l_1)\stuffle\pi_Y(l_2))\}_{l_1,l_2\in\Lyn X\setminus\{x_0\}}\subsetneq x_0\ncp{\Q}{X}x_1$ and $\{l_1\shuffle l_2-\pi_Y(\pi_X(l_1)\shuffle\pi_X(l_2))\}_{l_1,l_2\in\Lyn Y}\subsetneq(Y\setminus\{y_1\})\ncp{\Q}{Y}$, using \eqref{perrin}, $\Q$-algebraicly generate the kernels of the $\zeta$ polymorphism $\ker\zeta$ \cite{SLC43,FPSAC97}.

\item One can consider the \textit{double shuffle relations with simultaneous regularizations to the indeterminate $T$}, \textit{i.e.} $\zeta^T_{\shuffle}(x_1)=\zeta^T_{\stuffle}(y_1)=T$ (for the generators of weight $1$) and $\zeta^T_{\shuffle}(u\shuffle v)=\zeta^T_{\shuffle}(u)\zeta_{\shuffle}(v)$ (resp. $\zeta^T_{\stuffle}(u\stuffle v)=\zeta^T_{\stuffle}(u)\zeta^T_{\stuffle}(v)$), for $u$ and $v\in X^*$ (resp. $Y^*$) \cite{Cartier2,IharaKanekoZagier}. Since $T$ is transcendent then one can specialize $T=0$ or $T=\gamma$. So it was done in \cite{racinet,kaneko} with $T=0$. But the $\Q$-algebra $(\calZ,\times,1)$ is not a $\Q[T]$-algebra (see also Corollary \ref{notalgebraic}.)

\item In particular, for $T=0$, the \textit{regularized} morphisms $\zeta^0_{\shuffle}$ and $\zeta^0_{\stuffle}$ coincid, respectively,  with the morphisms $\zeta_{\shuffle}$ and $\zeta_{\stuffle}$ described in \eqref{regularization0}--\eqref{simul_regu}.

Moreover, \eqref{simul_regu} prove the \textit{existence} of a \textit{simultaneously regularizations} of the divergent polyzetas $\zeta(x_1)$ and $\zeta(y_1)$ to the same \textit{real} value, equal $0$, and then Proposition \ref{noyeaux} leads to  the \textit{regularized double shuffle relations} considered in \cite{racinet,kaneko}.
\end{enumerate}
\end{remark}

Now, using \eqref{gs}, the singularities analysis on $\{\Li_{w}\}_{w\in X^*}$ yield the following behaviors of the noncommutative generating series $\L$, put in the MRS form as in \eqref{Zshuffle} \cite{FPSAC98}
\begin{eqnarray}\label{behaviorsL} 
\L(z)\sim_0e^{x_0\log z}&\mbox{and}&\L(z)\sim_1e^{-x_1\log(1-z)}Z_{\shuffle},\end{eqnarray}
and then by a transfer theorem \cite{flajo_sedg} from $\{\Li_{w}\}_{w\in X^*}$ to $\{\H_{w}\}_{w\in Y^*}$, these lead also to the following behaviors of the noncommutative generating series $\H$ \cite{Daresbury,JSC}
\begin{eqnarray}\label{behaviorsH}
\H(n)\sim_{+\infty}e^{-\sum\limits_{k\ge1}\H_{y_k}(n){(-y_1)^k}/k}\pi_Y(Z_{\shuffle}).
\end{eqnarray}

By term by term differentiations (see \eqref{differentialforms} and \eqref{differentiation}), the noncommutative generating series $\L$ satisfies the following noncommutative differential equation \cite{acta,VJM,CM}
\begin{eqnarray}\label{NCDE}
d\L=({\omega_0}{x_0}+{\omega_1}{x_1})\L,
\end{eqnarray}
in which the Haussdorf group, \textit{i.e.}
$e^{\ncs{\calL ie_{\C}}{X}}=\{e^C\}_{C\in\ncs{\calL ie_{\C}}{X}}$,
plays the r\^ole of the differential Galois group of the grouplike solutions of \eqref{NCDE} and any grouplike series in the form $\L e^C$, where $C\in\ncs{\calL ie_{\C}}{X}$, satisfies \eqref{NCDE} \cite{acta,VJM,CM}.

A well known subgroup of this Haussdorf group is the monodromy group, around $0$ and $1$, of $\L$ \cite{FPSAC98}:
\begin{eqnarray}
\calM_0\L=\L e^{2\mathrm{i}\pi x_0}&\mbox{and}&
\calM_1\L=\L(Z_{\shuffle}e^{-2\mathrm{i}\pi x_1}Z_{\shuffle}^{-1}).
\end{eqnarray}
More generally, since $\Q\subset A\subset\C$ then one also has \cite{acta,CM}\begin{eqnarray}\label{dmA}
dm(A):=\{Z_{\shuffle}e^C\vert C\in\ncs{\calL ie_A}{X},
\scal{e^C}{x_0}=\scal{e^C}{x_1}=0\}\subsetneq e^{\ncs{\calL ie_{\C}}{X}}.
\end{eqnarray}
For any grouplike series, for $\Delta_{\shuffle}$, $\overline{Z}_{\shuffle}:=Z_{\shuffle}e^C\in dm(A)$ with $e^C\in\ncs{\calL ie_A}{X}$, satisfying $\scal{\overline{Z}_{\shuffle}}{x_0}=\scal{\overline{Z}_{\shuffle}}{x_1}=0$, let $\overline{\L}:=\L e^C$ and then
\begin{eqnarray}
\overline{\H}(n):=\sum_{w\in Y^*}\overline{\H}_{w}(n)w,&\mbox{where}&
\overline{\H}_{w}(n)=\scal{\frac{\scal{\overline{\L}(z)}{\pi_X(w)}}{1-z}}{z^n}.
\end{eqnarray}

\begin{remark}\label{associator}
\begin{enumerate}
\item $dm(A)$ contains $DM(A)$ introduced by Cartier in \cite{cartier2}.

\item In \cite{drinfeld2}, Drinfel'd stated that the noncommutative differential equation \eqref{NCDE} admits a unique solution, $G_0$ (resp. $G_1$), satisfying the following asymptotic condition
\begin{eqnarray*}
G_0(z)\sim_0e^{x_0\log(z)}=z^{x_0}&\mbox{(resp. }G_1(z)\sim_1e^{-x_1\log(1-z)}=(1-z)^{-x_1})
\end{eqnarray*}
and there is unique grouplike series $\Phi_{KZ}\in\ncs{\R}{X}$, called Drinfel'd series or Drinfel'd associator \cite{cartier2}, such that $G_0=G_1\Phi_{KZ}$ but neither have been constructed such expressions of $\Phi_{KZ}$ nor have been made explicit $G_0,G_1$.
We can recognize that $Z_{\shuffle}$, operated in \eqref{behaviorsL}, is nothing other than $\Phi_{KZ}$ (whereas $Z_{\gamma}$ and $Z_{\stuffle}$ are absent in \cite{drinfeld2}).
\end{enumerate}
\end{remark}

The behaviors in \eqref{behaviorsL}--\eqref{behaviorsH} yield the following renormalizations.

\begin{theorem}[Abel like theorem, \cite{Daresbury,JSC}]\label{renormalization1}
The following limits exist (see also \eqref{zetavalues})
\begin{eqnarray*}
\lim_{z\rightarrow 1}e^{y_1\log(1-z)}\pi_Y(\L(z))
=\lim_{n\rightarrow\infty}e^{\sum\limits_{k\ge1}\H_{y_k}(n){(-y_1)^k}/k}\H(n)
=\pi_Y(Z_{\shuffle}).
\end{eqnarray*}
\end{theorem}

\begin{corollary}[\cite{acta,VJM}]\label{pont}
Let $B(y_1):=e^{\gamma y_1-\sum\limits_{k\ge2}{\zeta(k)}{(-y_1)^k}/k}$
and\footnote{$B'(y_1)$ corresponds to the Ecalle's \texttt{Mono} mould \cite{ecalle}.} $B'(y_1):=e^{-\sum\limits_{k\ge2}{\zeta(k)}{(-y_1)^k}/k}$. Then
$Z_{\gamma}=B(y_1)\pi_Y({Z}_{\shuffle})$ and, by cancellation, it is equivalent to
$Z_{\stuffle}=B'(y_1)\pi_Y({Z}_{\shuffle})$.
\end{corollary}

\begin{example}[\cite{JSC}]
Extracting coefficients on $\L(z)\sim_1e^{-x_1\log(1-z)}Z_{\shuffle}$ then transferring,
\begin{eqnarray*}
\Li_{2,1}(z)&=&\zeta(3)+(1-z)\log(1-z)-(1-z)^{-1}-\frac12(1-z)\log^2(1-z)\\
		&+&(1-z)^2\biggl(-\frac14\log^2(1-z)+\frac14\log(1-z)\biggr)+\cdots,\cr
\H_{2,1}(n)&=&\zeta(3)-\frac1n(\log(n)+1+\gamma)+\frac1{2n}\log(n)+\cdots,\cr
\Li_{1,2}(z)&=&2-2\zeta(3)-\zeta(2)\log(1-z)-2(1-z)\log(1-z)\cr
		&+&(1-z)\log^2(1-z)+(1-z)^2\biggl(\frac12\log^2(1-z)-\frac12\log(1-z)\biggr)+\cdots,\cr
\H_{1,2}(n)&=&\zeta(2)\gamma-2\zeta(3)+\zeta(2)\log(n)+\frac1{2n}(\zeta(2)+2)+\cdots.
\end{eqnarray*}

Hence, in each convergent and divergent cases,
\begin{eqnarray}
&\mathop{\mathrm{f.p.}}\limits_{n\rightarrow+\infty}\H_{2,1}(n)
=\mathop{\mathrm{f.p.}}\limits_{z\rightarrow1}\Li_{2,1}(z)=\zeta(2,1)=\zeta(3),\\
&\mathop{\mathrm{f.p.}}\limits_{n\rightarrow+\infty}\H_{1,2}(n)=\zeta(2)\gamma-2\zeta(3)\neq
\mathop{\mathrm{f.p.}}\limits_{z\rightarrow1}\Li_{1,2}(z)=2-2\zeta(3),
\end{eqnarray}
with $\zeta(2)\gamma=0.94948171111498152454556410223170493364000594947366\cdots$.
\end{example}

Identifying $\{\scal{Z_{\gamma}}{y_1^kw}\}_{k\in\N_{\ge1}\atop w\in\CONV}$ in $Z_{\gamma}=B(y_1)\pi_Y({Z}_{\shuffle})$
(see Corollary \ref{pont}) and using the Bell polynomials $\{b_{n,k}(t_1,\ldots,t_k)\}_{n,k\in\N_+}$, the constants associated to
$\{\H_{y_1^kw}\}_{k\in\N_{\ge1}\atop w\in\CONV\cup\{1_{Y^*}\}}$ (depending on $\gamma$) are expressed as follows \cite{Daresbury,JSC}
\begin{eqnarray}
\gamma_{y_1^kw}&=&\sum_{i=0}^k\frac{\zeta(x_0(-x_1)^{k-i}\shuffle\pi_Xw])}{i!}\Big(\sum_{j=1}^ib_{i,j}(\gamma,-\zeta(2),2\zeta(3),\ldots)\Big),\label{y_1^kw}\\
\gamma_{y_1^k}&=&\sum_{{s_1,\ldots,s_k>0}\atop{s_1+\ldots+ks_k=k}}\frac{(-1)^k}{s_1!\ldots s_k!}
(-\gamma)^{s_1}\Big(-\frac{\zeta(2)}{2}\Big)^{s_2}\ldots\Big(-\frac{\zeta(k)}{k}\Big)^{s_k}.\label{y_1^k}
\end{eqnarray}

\begin{example}[Generalized Euler's gamma constant, \cite{Daresbury,JSC}]\label{generalizedgamma}
\begin{eqnarray*}
\gamma_{1,1}&=&\frac12(\gamma^2-\zeta(2)),\cr
\gamma_{1,1,1}&=&\frac16(\gamma^3-3\zeta(2)\gamma+2\zeta(3)),\cr
\gamma_{1,7}&=&\zeta(7)\gamma+\zeta(3)\zeta(5)-\frac{54}{175}\zeta(2)^4,\cr
\gamma_{1,1,6}&=&\frac{4}{35}\zeta(2)^3\gamma^2+(\zeta(2)\zeta(5)
+\frac25\zeta(3)\zeta(2)^2-4\zeta(7))\gamma+\frac{19}{35}\zeta(2)^4+\frac12\zeta(2)\zeta(3)^2\cr
&+&\zeta(6,2)-4\zeta(3)\zeta(5).
\end{eqnarray*}
\end{example}

But identifying the local coordinates, in $Z_{\gamma}=B(y_1)\pi_Y({Z}_{\shuffle})$ and $Z_{\shuffle}=B(x_1)^{-1}\pi_XZ_{\gamma}$, one obtains the $\Q$-algebraic relations, homogeneous in weight (see Remark \ref{length}) and independent on $\gamma$ (see also Corollary \ref{notalgebraic} below), among $\{\zeta(S_l)\}_{l\in\Lyn X\setminus X}$ and among $\{\zeta(\Sigma_l)\}_{l\in\Lyn Y\setminus\{y_1\}}$ \cite{Bui} (see also Remarks \ref{weightX} and \ref{weightY}).

\begin{example}[Polynomial relations on local coordinates, \cite{Bui}]\label{E1}\\
\small
\center
$\begin{array}{|c|rcl|rcl|}
		\hline
		&\mbox{Relations}\!\!&\text{on}&\!\!\{\zeta(\Sigma_l)\}_{l\in\Lyn Y\setminus\{y_1\}}&\mbox{Relations}\!\!&\text{on}&\!\!\{\zeta(S_l)\}_{l\in\Lyn X\setminus X}\\
		\hline
		3&\zeta(\Sigma_{y_2y_1})&=&\frac{3}{2}\zeta(\Sigma_{y_3})&\zeta(S_{x_0x_1^2})&=&\zeta(S_{x_0^2x_1})\\ 
		\hline
		&\zeta(\Sigma_{y_4})&=&\frac{2}{5}\zeta(\Sigma_{y_2})^2&\zeta(S_{x_0^3x_1})&=&\frac{2}{5}\zeta(S_{x_0x_1})^2\\               
		4&\zeta(\Sigma_{y_3y_1})&=&\frac{3}{10}\zeta(\Sigma_{y_2})^2&\zeta(S_{x_0^2x_1^2})&=&\frac{1}{10}\zeta(S_{x_0x_1})^2\\  
		&\zeta(\Sigma_{y_2y_1^2})&=&\frac{2}{3}\zeta(\Sigma_{y_2})^2&\zeta(S_{x_0x_1^3})&=&\frac{2}{5}\zeta(S_{x_0x_1})^2\\ 
		\hline
		&\zeta(\Sigma_{y_3y_2})&=&3\zeta(\Sigma_{y_3})\zeta(\Sigma_{y_2})-5\zeta(\Sigma_{y_5})&\zeta(S_{x_0^3x_1^2})&=&-\zeta(S_{x_0^2x_1})\zeta(S_{x_0x_1})+2\zeta(S_{x_0^4x_1})\\ 
		&\zeta(\Sigma_{y_4y_1})&=&-\zeta(\Sigma_{y_3})\zeta(\Sigma_{y_2})+\frac{5}{2}\zeta(\Sigma_{y_5})&\zeta(S_{x_0^2x_1x_0x_1})&=&-\frac{3}{2}\zeta(S_{x_0^4x_1})+\zeta(S_{x_0^2x_1})\zeta(S_{x_0x_1})\\   
		5&\zeta(\Sigma_{y_2^2y_1})&=&\frac{3}{2}\zeta(\Sigma_{y_3})\zeta(\Sigma_{y_2})-\frac{25}{12}\zeta(\Sigma_{y_5})&\zeta(S_{x_0^2x_1^3})&=&-\zeta(S_{x_0^2x_1})\zeta(S_{x_0x_1})+2\zeta(S_{x_0^4x_1})\\ 
		&\zeta(\Sigma_{y_3y_1^2})&=&\frac{5}{12}\zeta(\Sigma_{y_5})&\zeta(S_{x_0x_1x_0x_1^2})&=&\frac{1}{2}\zeta(S_{x_0^4x_1})\\ 
		&\zeta(\Sigma_{y_2y_1^3})&=&\frac{1}{4}\zeta(\Sigma_{y_3})\zeta(\Sigma_{y_2})+\frac{5}{4}\zeta(\Sigma_{y_5})&\zeta(S_{x_0x_1^4})&=&\zeta(S_{x_0^4x_1})\\ 
		\hline
		&\zeta(\Sigma_{y_6})&=&\frac{8}{35}\zeta(\Sigma_{y_2})^3& \zeta(S_{x_0^5x_1})&=&\frac{8}{35}\zeta(S_{x_0x_1})^3\\ 
		&\zeta(\Sigma_{y_4y_2})&=&\zeta(\Sigma_{y_3})^2-\frac{4}{21}\zeta(\Sigma_{y_2})^3& \zeta(S_{x_0^4x_1^2})&=&\frac{6}{35}\zeta(S_{x_0x_1})^3-\frac{1}{2}\zeta(S_{x_0^2x_1})^2\\ 
		&\zeta(\Sigma_{y_5y_1})&=&\frac{2}{7}\zeta(\Sigma_{y_2})^3-\frac{1}{2}\zeta(\Sigma_{y_3})^2& \zeta(S_{x_0^3x_1x_0x_1})&=&\frac{4}{105}\zeta(S_{x_0x_1})^3 \\ 
		&\zeta(\Sigma_{y_3y_1y_2})&=&-\frac{17}{30}\zeta(\Sigma_{y_2})^3+\frac{9}{4}\zeta(\Sigma_{y_3})^2&\zeta(S_{x_0^3x_1^3})&=&\frac{23}{70}\zeta(S_{x_0x_1})^3-\zeta(S_{x_0^2x_1})^2\\ 
		6&\zeta(\Sigma_{y_3y_2y_1})&=&3\zeta(\Sigma_{y_3})^2-\frac{9}{10}\zeta(\Sigma_{y_2})^3& \zeta(S_{x_0^2x_1x_0x_1^2})&=&\frac{2}{105}\zeta(S_{x_0x_1})^3\\ 
		&\zeta(\Sigma_{y_4y_1^2})&=&\frac{3}{10}\zeta(\Sigma_{y_2})^3-\frac{3}{4}\zeta(\Sigma_{y_3})^2&\zeta(S_{x_0^2x_1^2x_0x_1})&=&-\frac{89}{210}\zeta(S_{x_0x_1})^3+\frac{3}{2}\zeta(S_{x_0^2x_1})^2 \\ 
		&\zeta(\Sigma_{y_2^2y_1^2})&=&\frac{11}{63}\zeta(\Sigma_{y_2})^3-\frac{1}{4}\zeta(\Sigma_{y_3})^2& \zeta(S_{x_0^2x_1^4})&=&\frac{6}{35}\zeta(S_{x_0x_1})^3-\frac{1}{2}\zeta(S_{x_0^2x_1})^2\\ 
		&\zeta(\Sigma_{y_3y_1^3})&=&\frac{1}{21}\zeta(\Sigma_{y_2})^3&\zeta(S_{x_0x_1x_0x_1^3})&=&\frac{8}{21}\zeta(S_{x_0x_1})^3-\zeta(S_{x_0^2x_1})^2 \\ 
		&\zeta(\Sigma_{y_2y_1^4})&=&\frac{17}{50}\zeta(\Sigma_{y_2})^3+\frac{3}{16}\zeta(\Sigma_{y_3})^2&\zeta(S_{x_0x_1^5})&=&\frac{8}{35}\zeta(S_{x_0x_1})^3\\ 
	\hline
	\end{array}$
\end{example}

As solution of \eqref{NCDE} and, as in \eqref{behaviorsL}--\eqref{behaviorsH}, $\overline{\L}$ has the following behaviors \cite{acta,VJM,CM}
\begin{eqnarray}
\overline{\L}(z)\sim_0e^{x_0\log z}e^C&\mbox{and}&
\overline{\L}(z)\sim_1e^{-x_1\log(1-z)}\overline{Z}_{\shuffle},\label{behaviorsLE}\\
&&\overline{\H}(n)\sim_{+\infty}e^{-\sum\limits_{k\ge1}\H_{y_k}(n){(-y_1)^k}/k}\pi_Y(\overline Z_{\shuffle}).\label{behaviorsLEcons}
\end{eqnarray}
These yield similar results as in Theorem \ref{renormalization1} and in Corollary \ref{pont} \cite{acta,VJM,CM}:
\begin{eqnarray}
\lim_{z\rightarrow 1}e^{y_1\log(1-z)}\pi_Y(\overline{\L}(z))
=&\pi_Y(\overline{Z}_{\shuffle})&=\lim_{n\rightarrow\infty}
e^{\sum\limits_{k\ge1}\H_{y_k}(n){(-y_1)^k}/k}\overline{\H}(n),\\
\overline{Z}_{\gamma}=B(y_1)\pi_Y(\overline{Z}_{\shuffle})&\iff&
\overline{Z}_{\stuffle}=B'(y_1)\pi_Y(\overline{Z}_{\shuffle}).
\end{eqnarray}
Hence, $\L$ is the unique solution of \eqref{NCDE} satisfying the asymptotic conditions in \eqref{behaviorsL}, \textit{i.e.} with $e^C=1_{X^*}$, and then $Z_{\shuffle}$ is unique \cite{CM}. Identifying the local coordinates, in
\begin{eqnarray}
\overline{Z}_{\gamma}=B(y_1)\pi_Y({\overline{Z}}_{\shuffle})&\mbox{and}&
\overline{Z}_{\shuffle}=B(x_1)^{-1}\pi_X\overline{Z}_{\gamma},
\end{eqnarray}
one also obtains algebraic relations among $\{\zeta(w)\}_{w\in\CONV}$, over $A$. Moreover,

\begin{corollary}[\cite{acta,VJM,CM}]\label{notalgebraic}
For any commutative $\Q$-algebra $A$, if there exists $A$-algebraic relations among $\{\zeta(w)\}_{w\in\CONV}$ and $\gamma$ then $\gamma\in A$.
\end{corollary}

\section{From algebraic relations among polyzetas to rewriting systems}\label{Algorithm}
The previous relations among polyzetas obtained in Section \ref{Indexation} provide the polynomials $\{Q_l\}_{l\in\Lyn\calX\setminus\gDIV}$ generating inside $\ker\zeta$ (see \eqref{Q_XQ_Y}--\eqref{R_XR_Y} below). Replacing ``$=$'' by ``$\rightarrow$'', they yield the rewriting rules (see \eqref{R_XR_Y}--\eqref{Rirr} below) endowed with the increasing sets of irreducible polyzetas, $\{\calZ_{irr}^{\calX,\le p}\}_{p\ge2}$, as well as their images, $\{\calL_{irr}^{\calX,\le p}\}_{p\ge2}$, by a section of $\zeta$ such that the following restriction of $\zeta$ polymorphism is \textit{bijective} \cite{VJM,CM}
\begin{eqnarray}
&\zeta:\left\{(\Q[\calL^{X,\infty}_{irr}],\shuffle,1_{X^*})\atop(\Q[\calL^{Y,\infty}_{irr}],\stuffle,1_{Y^*})\right\}\hooklongtwoheadrightarrow
(\Q[\calZ^{\calX,\infty}_{irr}],\times,1)\cong(\calZ,\times,1),&\label{restriction}\\
&\calZ_{irr}^{\calX,\le 2}\subset\cdots\subset\calZ_{irr}^{\calX,\le p}\subset\cdots\subset\calZ_{irr}^{\calX,\infty}:=\mathop{\cup}\limits_{p\ge2}\calZ_{irr}^{\calX,\le p},&\label{Zirr}\\
&\calL_{irr}^{\calX,\le 2}\subset\cdots\subset\calL_{irr}^{\calX,\le p}\subset\cdots\subset\calL_{irr}^{\calX,\infty}:=\mathop{\cup}\limits_{p\ge2}\calL_{irr}^{\calX,\le p}.&\label{Lirr}
\end{eqnarray}

\begin{remark}\label{bijection}
By the restriction in \eqref{restriction}--\eqref{Lirr}, one has (see Examples \ref{E}--\ref{E0} below)
\begin{eqnarray*}
S_l\in\calL^{X,\infty}_{irr}\subsetneq\{S_l\}_{l\Lyn X\setminus X}&\iff&
\zeta(S_l)\in\calZ^{X,\infty}_{irr}\subsetneq\{\zeta(S_l)\}_{S_l\in\Lyn X\setminus X},\\
\Sigma_l\in\calL^{Y,\infty}_{irr}\subsetneq\{\Sigma_l\}_{l\Lyn Y\setminus\{y_1\}}&\iff&
\zeta(\Sigma_l)\in\calZ^{X,\infty}_{irr}\subsetneq\{\zeta(\Sigma_l)\}_{l\in\Lyn Y\setminus\{y_1\}},
\end{eqnarray*}
meaning that the elements of $\calZ_{irr}^{X,\infty}$ (resp. $\calZ_{irr}^{Y,\infty}$) are $\Q$-algebraic generators of $\calZ$ (see also Proposition \ref{directsum} below). 
\end{remark}

\begin{example}[Homogeneous in weight polynomials\footnote{This table is ``handly'' obtained from the one of Example \ref{E1} in \cite{Bui}.} generating inside $\ker\zeta$]\label{E3}
\small
\center
$\begin{array}{|c|r|r|}
	\hline
		&\{Q_l\}_{l\in\Lyn Y\setminus\{y_1\}}&\{Q_l\}_{l\in\Lyn X\setminus X}\\
		\hline
		3&\zeta({\Sigma_{y_2y_1}-\frac{3}{2}\Sigma_{y_3}})=0&\zeta({S_{x_0x_1^2}-S_{x_0^2x_1}})=0\\ 
		\hline
		&\zeta({\Sigma_{y_4}-\frac{2}{5}\Sigma_{y_2}^{\stuffle2}})=0&\zeta({S_{x_0^3x_1}-\frac{2}{5}S_{x_0x_1}^{\shuffle2}})=0\\          
		4&\zeta({\Sigma_{y_3y_1}-\frac{3}{10}\Sigma_{y_2}^{\stuffle2}})=0&\zeta({S_{x_0^2x_1^2}-\frac{1}{10}S_{x_0x_1}^{\shuffle2}})=0\\ 
		&\zeta({\Sigma_{y_2y_1^2}-\frac{2}{3}\Sigma_{y_2}^{\stuffle2}})=0&\zeta({S_{x_0x_1^3}-\frac{2}{5}S_{x_0x_1}^{\shuffle2}})=0\\ 
		\hline
		&\zeta({\Sigma_{y_3y_2}-3\Sigma_{y_3}\stuffle\Sigma_{y_2}-5\Sigma_{y_5}})=0&\zeta({S_{x_0^3x_1^2}-S_{x_0^2x_1}\shuffle S_{x_0x_1}+2S_{x_0^4x_1}})=0\\ 
		&\zeta({\Sigma_{y_4y_1}-\Sigma_{y_3}\stuffle\Sigma_{y_2})+\frac{5}{2}\Sigma_{y_5}})=0&\zeta({S_{x_0^2x_1x_0x_1}-\frac{3}{2}S_{x_0^4x_1}+S_{x_0^2x_1}\shuffle S_{x_0x_1}})=0\\ 
		5&\zeta({\Sigma_{y_2^2y_1}-\frac{3}{2}\Sigma_{y_3}\stuffle\Sigma_{y_2}-\frac{25}{12}\Sigma_{y_5}})=0&\zeta({S_{x_0^2x_1^3}-S_{x_0^2x_1}\shuffle S_{x_0x_1}+2S_{x_0^4x_1}})=0\\ 
		&\zeta({\Sigma_{y_3y_1^2}-\frac{5}{12}\Sigma_{y_5}})=0&\zeta({S_{x_0x_1x_0x_1^2}-\frac{1}{2}S_{x_0^4x_1}})=0\\ 
		&\zeta({\Sigma_{y_2y_1^3}-\frac{1}{4}\Sigma_{y_3}\stuffle\Sigma_{y_2})+\frac{5}{4}\Sigma_{y_5}})=0&\zeta({S_{x_0x_1^4}-S_{x_0^4x_1}})=0\\ 
		\hline
		&\zeta({\Sigma_{y_6}-\frac{8}{35}\Sigma_{y_2}^{\stuffle3}})=0&\zeta({S_{x_0^5x_1}-\frac{8}{35}S_{x_0x_1}^{\shuffle3}})=0\\ 
		&\zeta({\Sigma_{y_4y_2}-\Sigma_{y_3}^{\stuffle2}-\frac{4}{21}\Sigma_{y_2}^{\stuffle3}})=0&\zeta({S_{x_0^4x_1^2}-\frac{6}{35}S_{x_0x_1}^{\shuffle3}-\frac{1}{2}S_{x_0^2x_1}^{\shuffle2}})=0\\ 
		&\zeta({\Sigma_{y_5y_1}-\frac{2}{7}\Sigma_{y_2}^{\stuffle3}-\frac{1}{2}\Sigma_{y_3}^{\stuffle2}})=0&\zeta({S_{x_0^3x_1x_0x_1}-\frac{4}{105}S_{x_0x_1}^{\shuffle3}})=0\\ 
		&\zeta({\Sigma_{y_3y_1y_2}-\frac{17}{30}\Sigma_{y_2}^{\stuffle3}+\frac{9}{4}\Sigma_{y_3}^{\stuffle2}})=0&\zeta({S_{x_0^3x_1^3}-\frac{23}{70}S_{x_0x_1}^{\shuffle3}-S_{x_0^2x_1}^{\shuffle2}})=0\\ 
		6&\zeta({\Sigma_{y_3y_2y_1}-3\Sigma_{y_3}^{\stuffle2}-\frac{9}{10}\Sigma_{y_2}^{\stuffle3}})=0&\zeta({S_{x_0^2x_1x_0x_1^2}-\frac{2}{105}S_{x_0x_1}^{\shuffle3}})=0\\ 
		&\zeta({\Sigma_{y_4y_1^2}-\frac{3}{10}\Sigma_{y_2}^{\stuffle2}-\frac{3}{4}\Sigma_{y_3}^{\stuffle2}})=0&\zeta({S_{x_0^2x_1^2x_0x_1}-\frac{89}{210}S_{x_0x_1}^{\shuffle3}+\frac{3}{2}S_{x_0^2x_1}^{\shuffle2}})=0\\ 
		&\zeta({\Sigma_{y_2^2y_1^2}-\frac{11}{63}\Sigma_{y_2}^{\stuffle2}-\frac{1}{4}\Sigma_{y_3}^{\stuffle2}})=0&\zeta({S_{x_0^2x_1^4}-\frac{6}{35}S_{x_0x_1}^{\shuffle3}-\frac{1}{2}S_{x_0^2x_1}^{\shuffle2}})=0\\ 
		&\zeta({\Sigma_{y_3y_1^3}-\frac{1}{21}\Sigma_{y_2}^{\stuffle3}})=0&\zeta({S_{x_0x_1x_0x_1^3}-\frac{8}{21}S_{x_0x_1}^{\shuffle3}-S_{x_0^2x_1}^{\shuffle2}})=0\\ 
		&\zeta({\Sigma_{y_2y_1^4}-\frac{17}{50}\Sigma_{y_2}^{\stuffle3}+\frac{3}{16}\Sigma_{y_3}^{\stuffle2}})=0&\zeta({S_{x_0x_1^5}-\frac{8}{35}S_{x_0x_1}^{\shuffle3}})=0\\ 
	\hline
	\end{array}$
\end{example}

\begin{example}[Rewitting among local coordinates\footnote{This table is ``handly'' obtained from the one of Example \ref{E1} in \cite{Bui}.} and irreducible terms, \cite{Bui}]\label{E}\\
\small
\center
$\begin{array}{|c|rcl|rcl|}
		\hline
		&\mbox{Rewriting}\!\!&\text{on}&\!\!\{\zeta(\Sigma_l)\}_{l\in\Lyn Y\setminus\{y_1\}}&\mbox{Rewriting}\!\!&\text{on}&\!\!\{\zeta(S_l)\}_{l\in\Lyn X\setminus X}\\
		\hline
		3&\zeta(\Sigma_{y_2y_1})&\to&\frac{3}{2}\zeta(\Sigma_{y_3})&\zeta(S_{x_0x_1^2})&\to&\zeta(S_{x_0^2x_1})\\ 
		\hline
		&\zeta(\Sigma_{y_4})&\to&\frac{2}{5}\zeta(\Sigma_{y_2})^2&\zeta(S_{x_0^3x_1})&\to&\frac{2}{5}\zeta(S_{x_0x_1})^2\\               
		4&\zeta(\Sigma_{y_3y_1})&\to&\frac{3}{10}\zeta(\Sigma_{y_2})^2&\zeta(S_{x_0^2x_1^2})&\to&\frac{1}{10}\zeta(S_{x_0x_1})^2\\  
		&\zeta(\Sigma_{y_2y_1^2})&\to&\frac{2}{3}\zeta(\Sigma_{y_2})^2&\zeta(S_{x_0x_1^3})&\to&\frac{2}{5}\zeta(S_{x_0x_1})^2\\ 
		\hline
		&\zeta(\Sigma_{y_3y_2})&\to&3\zeta(\Sigma_{y_3})\zeta(\Sigma_{y_2})-5\zeta(\Sigma_{y_5})&\zeta(S_{x_0^3x_1^2})&\to&-\zeta(S_{x_0^2x_1})\zeta(S_{x_0x_1})+2\zeta(S_{x_0^4x_1})\\ 
		&\zeta(\Sigma_{y_4y_1})&\to&-\zeta(\Sigma_{y_3})\zeta(\Sigma_{y_2})+\frac{5}{2}\zeta(\Sigma_{y_5})&\zeta(S_{x_0^2x_1x_0x_1})&\to&-\frac{3}{2}\zeta(S_{x_0^4x_1})+\zeta(S_{x_0^2x_1})\zeta(S_{x_0x_1})\\   
		5&\zeta(\Sigma_{y_2^2y_1})&\to&\frac{3}{2}\zeta(\Sigma_{y_3})\zeta(\Sigma_{y_2})-\frac{25}{12}\zeta(\Sigma_{y_5})&\zeta(S_{x_0^2x_1^3})&\to&-\zeta(S_{x_0^2x_1})\zeta(S_{x_0x_1})+2\zeta(S_{x_0^4x_1})\\ 
		&\zeta(\Sigma_{y_3y_1^2})&\to&\frac{5}{12}\zeta(\Sigma_{y_5})&\zeta(S_{x_0x_1x_0x_1^2})&\to&\frac{1}{2}\zeta(S_{x_0^4x_1})\\ 
		&\zeta(\Sigma_{y_2y_1^3})&\to&\frac{1}{4}\zeta(\Sigma_{y_3})\zeta(\Sigma_{y_2})+\frac{5}{4}\zeta(\Sigma_{y_5})&\zeta(S_{x_0x_1^4})&\to&\zeta(S_{x_0^4x_1})\\ 
		\hline
		&\zeta(\Sigma_{y_6})&\to&\frac{8}{35}\zeta(\Sigma_{y_2})^3& \zeta(S_{x_0^5x_1})&\to&\frac{8}{35}\zeta(S_{x_0x_1})^3\\ 
		&\zeta(\Sigma_{y_4y_2})&\to&\zeta(\Sigma_{y_3})^2-\frac{4}{21}\zeta(\Sigma_{y_2})^3& \zeta(S_{x_0^4x_1^2})&\to&\frac{6}{35}\zeta(S_{x_0x_1})^3-\frac{1}{2}\zeta(S_{x_0^2x_1})^2\\ 
		&\zeta(\Sigma_{y_5y_1})&\to&\frac{2}{7}\zeta(\Sigma_{y_2})^3-\frac{1}{2}\zeta(\Sigma_{y_3})^2& \zeta(S_{x_0^3x_1x_0x_1})&\to&\frac{4}{105}\zeta(S_{x_0x_1})^3 \\ 
		&\zeta(\Sigma_{y_3y_1y_2})&\to&-\frac{17}{30}\zeta(\Sigma_{y_2})^3+\frac{9}{4}\zeta(\Sigma_{y_3})^2&\zeta(S_{x_0^3x_1^3})&\to&\frac{23}{70}\zeta(S_{x_0x_1})^3-\zeta(S_{x_0^2x_1})^2\\ 
		6&\zeta(\Sigma_{y_3y_2y_1})&\to&3\zeta(\Sigma_{y_3})^2-\frac{9}{10}\zeta(\Sigma_{y_2})^3& \zeta(S_{x_0^2x_1x_0x_1^2})&\to&\frac{2}{105}\zeta(S_{x_0x_1})^3\\ 
		&\zeta(\Sigma_{y_4y_1^2})&\to&\frac{3}{10}\zeta(\Sigma_{y_2})^3-\frac{3}{4}\zeta(\Sigma_{y_3})^2&\zeta(S_{x_0^2x_1^2x_0x_1})&\to&-\frac{89}{210}\zeta(S_{x_0x_1})^3+\frac{3}{2}\zeta(S_{x_0^2x_1})^2 \\ 
		&\zeta(\Sigma_{y_2^2y_1^2})&\to&\frac{11}{63}\zeta(\Sigma_{y_2})^3-\frac{1}{4}\zeta(\Sigma_{y_3})^2& \zeta(S_{x_0^2x_1^4})&\to&\frac{6}{35}\zeta(S_{x_0x_1})^3-\frac{1}{2}\zeta(S_{x_0^2x_1})^2\\ 
		&\zeta(\Sigma_{y_3y_1^3})&\to&\frac{1}{21}\zeta(\Sigma_{y_2})^3&\zeta(S_{x_0x_1x_0x_1^3})&\to&\frac{8}{21}\zeta(S_{x_0x_1})^3-\zeta(S_{x_0^2x_1})^2 \\ 
		&\zeta(\Sigma_{y_2y_1^4})&\to&\frac{17}{50}\zeta(\Sigma_{y_2})^3+\frac{3}{16}\zeta(\Sigma_{y_3})^2&\zeta(S_{x_0x_1^5})&\to&\frac{8}{35}\zeta(S_{x_0x_1})^3\\ 
	\hline
	\end{array}$
\begin{eqnarray*}
\calZ_{irr}^{X,\le12}&=&
\{\zeta(S_{x_0x_1}),\zeta(S_{x_0^2x_1}),\zeta(S_{x_0^4x_1}),\zeta(S_{x_0^6x_1}),\zeta(S_{x_0x_1^2x_0x_1^4}),\zeta(S_{x_0^8x_1}),\cr
&&\zeta(S_{x_0x_1^2x_0x_1^6}),\zeta(S_{x_0^{10}x_1}),\zeta(S_{x_0x_1^2x_0x_1^7}),\zeta(S_{x_0x_1^2x_0x_1^8}),\zeta(S_{x_0x_1^4x_0x_1^6})\}.\cr
\calZ_{irr}^{Y,\le12}&=&
\{\zeta(\Sigma_{y_2}),\zeta(\Sigma_{y_3}),\zeta(\Sigma_{y_5}),\zeta(\Sigma_{y_7}),\zeta(\Sigma_{y_3y_1^5}),\zeta(\Sigma_{y_9}),\zeta(\Sigma_{y_3y_1^7}),\cr
&&\zeta(\Sigma_{y_{11}}),\zeta(\Sigma_{y_2y_1^9}),\zeta(\Sigma_{y_3y_1^9}),\zeta(\Sigma_{y_2^ 2y_1^8})\}.
\end{eqnarray*}
\end{example}

Now, let us describ the algorithm {\bf LocalCoordinateIdentification} (partially implemented in \cite{Bui} and briefly described in \cite{LT12,Tokyo}), applying Theorem \ref{renormalization1} and Corollary \ref{pont}, bringing aditional informations and results to \cite{Bui}. It uses the notations concerning
\begin{enumerate}
\item the shuffle and quasi-shuffle subalgebras, recalled in Section \ref{Combinatorics}, \textit{i.e.}
\begin{eqnarray}
(\Q1_{X^*}\oplus x_0\ncp{\Q}{X}x_1,\shuffle,1_{X^*})&\cong&(\Q[\{S_l\}_{l\in\Lyn X\setminus X}],\shuffle,1_{X^*}),\label{subalgebra1}\\
(\Q1_{Y^*}\oplus(Y\setminus\{y_1\})\ncp{\Q}{Y},\stuffle,1_{Y^*})&\cong&(\Q[\{\Sigma_l\}_{l\in\Lyn Y\setminus\{y_1\}}],\stuffle,1_{Y^*}),\label{subalgebra2}
\end{eqnarray}
\item the algebras of polyzetas, recalled in Section \ref{Indexation}, \textit{i.e.}
\begin{eqnarray}
(\Q[\{\zeta(S_l)\}_{l\in\Lyn X\setminus X}],\times,1)&\mbox{and}&
(\Q[\{\zeta(\Sigma_l)\}_{l\in\Lyn Y\setminus\{y_1\}}],\times,1),
\end{eqnarray}
\item the rewriting systems, introduced in \eqref{Q_XQ_Y}--\eqref{Rirr} bellow, \textit{i.e.}
\begin{eqnarray}\label{RS}
(\Q1_{X^*}\oplus x_0\ncp{\Q}{X}x_1,\calR^X_{irr})&\mbox{and}&
(\Q1_{Y^*}\oplus(Y\setminus\{y_1\})\ncp{\Q}{Y},\calR^Y_{irr}),
\end{eqnarray}
which are being without critical pairs, noetherian, confluent.
\end{enumerate}

This algorithm comprises an initialization phase followed by two nested loops. The first loop focuses on the weight of Lyndon words, and the second on the ordered set of Lyndon words at the same weight. The last yields equations on polyzetas, as in \cite{Bui}, and then transforms them to corresponding rewriting rules or, exclusively, irreducible terms. This algorithm is illustrated by Examples \ref{E1}--\ref{E0}, provides \eqref{RS} and precisely contains
\begin{enumerate}
\item The above increasing set  $\{\calL_{irr}^{\calX,\le p}\}_{p\ge2}$ and its image, $\{\calZ_{irr}^{\calX,\le p}\}_{p\ge2}$, by the restriction of the $\zeta$ polymorphism, described in \eqref{restriction} which is bijective.

\item The following set of homogeneous in weight polynomials (see Remarks \ref{weightX} and \ref{weightY}) being image by a section of $\zeta$ from $\{\zeta(Q_l)=0\}_{l\in\Lyn\calX\setminus\gDIV}$:
\begin{eqnarray}\label{Q_XQ_Y}
\calQ_{\calX}=\{Q_l\}_{l\in\Lyn\calX\setminus\gDIV}.
\end{eqnarray}

\item The shuffle or quasi-shuffle $\Q$-ideal $\calR_{\calX}$ inside $\ker\zeta$, generated by $\calQ_{\calX}$ as follows
\begin{eqnarray}\label{R_XR_Y}
\calR_X:=(\mathrm{span}_{\Q}\calQ_X,\shuffle,1_{X^*})\subseteq\ker\zeta
&\mbox{and}&
\calR_Y:=(\mathrm{span}_{\Q}\calQ_Y,\stuffle,1_{Y^*})\subseteq\ker\zeta.
\end{eqnarray}

\item By \eqref{Q_XQ_Y}, for any $l\in\Lyn^p\calX:=\{l\in\Lyn\calX|(l)=p\ge2\}$, any nonzero homogeneous in weight polynomial $Q_l=\Sigma_l-\Upsilon_l$ (resp. $Q_l=S_l-U_l$) is
\begin{enumerate}
\item led by $\Sigma_l$ (resp. $S_l$), which does not belong to $\Q[\calL_{irr}^{\calX,\le p}]$,

\item followed by $-\Upsilon_l$ (resp. $-U_l$), which is canonically represented in $\Q[\calL_{irr}^{\calX,\le p}]$.
\end{enumerate}
Then let $\Sigma_l\rightarrow\Upsilon_l$ and $S_l\rightarrow U_l$ be the rewriting rules, respectively, of
\begin{eqnarray}\label{Rirr}
\calR_{irr}^Y:=\{\Sigma_l\rightarrow\Upsilon_l\}_{l\in\Lyn Y\setminus\{y_1\}}
&\mbox{and}&\calR_{irr}^X:=\{S_l\rightarrow U_l\}_{l\in\Lyn X\setminus X}.
\end{eqnarray}	

\item On the other hand, denoting $\calL_{irr}^{\calX,\le p}:=\calL_{irr}^{\calX,2}\cup\cdots\cup\calL_{irr}^{\calX,p}$, for any $l\in\Lyn\calX\setminus\gDIV$, the following assertions are equivalent (see Example \ref{E0} below)
\begin{enumerate}
\item $Q_l=0$,
\item $\Sigma_l\in\calL_{irr}^{Y,\le p}$ (resp. $S_l\in\calL_{irr}^{X,\le p}$),
\item $\Sigma_l\rightarrow\Sigma_l$ (resp. $S_l\rightarrow S_l$).
\end{enumerate}
\end{enumerate}

\noindent
\fbox{%
\begin{minipage}{1\textwidth}
{\bf LocalCoordinateIdentification}\\
$\calZ_{irr}^{Y,\infty}:=\{\};\calL_{irr}^{Y,\infty}:=\{\};\calR_{irr}^{Y}:=\{\};\calQ_{Y}:=\{\}$;\\
$\calZ_{irr}^{X,\infty}:=\{\};\calL_{irr}^{X,\infty}:=\{\};\calR_{irr}^{X}:=\{\};\calQ_{X}:=\{\}$;\\
for $p$ ranges in $2,\ldots,\infty$ do\\
\phantom{Here}\hspace{3mm}for $l$ ranges in the totally ordered $\Lyn^p Y$ do\\
\phantom{Here}\hspace{6mm}$\lambda:=\pi_X(l)$;\\
\phantom{Here}\hspace{6mm}identify $\scal{Z_{\gamma}}{\Pi_l}$ in $Z_{\gamma}=B(y_1)\pi_YZ_{\shuffle}$ and $\scal{Z_{\shuffle}}{P_\lambda}$
in $Z_{\shuffle}=B(x_1)^{-1}\pi_XZ_{\gamma}$;\\
\phantom{Here}\hspace{6mm}by elimination, obtain equations on $\{\zeta(\Sigma_{l'})\}_{l'\in\Lyn^pY\atop l'\preceq l}$
and on $\{\zeta(S_{\lambda'})\}_{\lambda'\in\Lyn^pX\atop\lambda'\preceq\lambda}$;\\
\phantom{Here}\hspace{6mm}express\footnote{This step and the following ones are not yet been achieved by the implementation in \cite{Bui}.} these equations, led by $\zeta(\Sigma_l)$ and by $\zeta(S_\lambda)$, as rewriting rules;\\
\phantom{Here}\hspace{6mm}if $\zeta(\Sigma_l)\rightarrow\zeta(\Sigma_l)$ then $\calZ_{irr}^{Y,\infty}:=\calZ_{irr}^{Y,\infty}\cup\{\zeta(\Sigma_l)\}$
and $\calL_{irr}^{Y,\infty}:=\calL_{irr}^{Y,\infty}\cup\{\Sigma_l\}$\\
\phantom{Here}\hspace{9mm}else $\calR_{irr}^Y:=\calR_{irr}^Y\cup\{\Sigma_l\rightarrow\Upsilon_l\}$
 and $\calQ_Y:=\calQ_Y\cup\{\Sigma_l-\Upsilon_l\}$;\\
\phantom{Here}\hspace{6mm}if $\zeta(S_\lambda)\rightarrow\zeta(S_\lambda)$
then $\calZ_{irr}^{X,\infty}:=\calZ_{irr}^{X,\infty}\cup\{\zeta(S_\lambda)\}$
and $\calL_{irr}^{X,\infty}:=\calL_{irr}^{X,\infty}\cup\{S_\lambda\}$\\
\phantom{Here}\hspace{9mm}else $\calR_{irr}^X:=\calR_{irr}^X\cup\{S_\lambda\rightarrow U_\lambda\}$ and $\calQ_X:=\calQ_X\cup\{S_\lambda-U_\lambda\}$\\
\phantom{Here}\hspace{3mm}end\_for\\
end\_for
\end{minipage}
}

In the other words, the ordering over the set of Lyndon words $\Lyn\calX$ induces the orderings over $\Q$-algebraic bases $\{S_\lambda\}_{\lambda\in\Lyn X\setminus X}$ and $\{\Sigma_l\}_{l\in\Lyn Y\setminus\{y_1\}}$ (see Remarks \ref{ordered1} and \ref{ordered2}), over the set of irreducible terms $\calL^{\calX,\infty}_{irr}$ and the $\Q$-ideal $\calR_{\calX}$ inside $\ker\zeta$, and then over the set of rewriting rules $\calR_{irr}^{\calX}$.Moreover,
\begin{enumerate}
\item each irreducible term, belonging to $\calL^{X,\infty}_{irr}$ (resp. $\calL^{Y,\infty}_{irr}$), is an element of the $\Q$-algebraic basis $\{S_l\}_{l\in\Lyn X\setminus X}$ (resp. $\{\Sigma_l\}_{l\in\Lyn Y\setminus\{y_1\}}$) of the subalgebra given in \eqref{subalgebra1} (resp. \eqref{subalgebra2}),

\item each rewriting rule, belonging to $\calR^{\calX}_{irr}$, admits the left side doing not belong to the $\Q$-algebra $\Q[\calL_{irr}^{{\calX},\infty}]$ and the right side being canonically represented in the $\Q$-algebra $\Q[\calL_{irr}^{{\calX},\infty}]$. The difference of these two sides belongs to the totally ordered ideal $\calR_{\calX}$ of the $\Q$-algebra $\Q[\Lyn\calX\setminus\gDIV]$.
\end{enumerate}

\begin{example}[Rewriting among $\{\Sigma_l\}_{l\in\Lyn Y\setminus\{y_1\}}$
and $\{S_l\}_{l\in\Lyn X\setminus X}$, irreducible terms\footnote{There is some typographical errors in \cite{LT12,Tokyo}.} \cite{LT12,Tokyo}]\label{E0}
\small
\center
$\begin{array}{|c|rcl|rcl|}
		\hline
		&\mbox{Rewriting}&\text{on}&\{\Sigma_l\}_{l\in\Lyn Y\setminus\{y_1\}}&\mbox{Rewriting}&\text{on}&\{S_l\}_{\Lyn X\setminus X}\cr 
		\hline
		3&\Sigma_{y_2y_1}&\rightarrow&\frac{3}{2}\Sigma_{y_3}&S_{x_0x_1^2}&\rightarrow&S_{x_0^2x_1}\cr 
		\hline
		&\Sigma_{y_4}&\rightarrow&\frac{2}{5}\Sigma_{y_2}^{\stuffle 2}&S_{x_0^3x_1}&\rightarrow&\frac{2}{5}S_{x_0x_1}^{\shuffle 2}\cr               
		4&\Sigma_{y_3y_1}&\rightarrow&\frac{3}{10}\Sigma_{y_2}^{\stuffle 2}&S_{x_0^2x_1^2}&\rightarrow&\frac{1}{10}S_{x_0x_1}^{\shuffle 2}\cr  
		&\Sigma_{y_2y_1^2}&\rightarrow&\frac{2}{3}\Sigma_{y_2}^{\stuffle 2}&S_{x_0x_1^3}&\rightarrow&\frac{2}{5}S_{x_0x_1}^{\shuffle 2}\cr 
		\hline
		&\Sigma_{y_3y_2}&\rightarrow&3\Sigma_{y_3}\stuffle\Sigma_{y_2}-5\Sigma_{y_5}&S_{x_0^3x_1^2}&\rightarrow&-S_{x_0^2x_1}\shuffle S_{x_0x_1}+2S_{x_0^4x_1}\cr 
		&\Sigma_{y_4y_1}&\rightarrow&-\Sigma_{y_3}\stuffle\Sigma_{y_2}+\frac{5}{2}\Sigma_{y_5}&S_{x_0^2x_1x_0x_1}&\rightarrow&-\frac{3}{2}S_{x_0^4x_1}+S_{x_0^2x_1}\shuffle S_{x_0x_1}\cr   
		5&\Sigma_{y_2^2y_1}&\rightarrow&\frac{3}{2}\Sigma_{y_3}\stuffle\Sigma_{y_2}-\frac{25}{12}\Sigma_{y_5}&S_{x_0^2x_1^3}&\rightarrow&-S_{x_0^2x_1}\shuffle S_{x_0x_1}+2S_{x_0^4x_1}\cr 
		&\Sigma_{y_3y_1^2}&\rightarrow&\frac{5}{12}\Sigma_{y_5}&S_{x_0x_1x_0x_1^2}&\rightarrow&\frac{1}{2}S_{x_0^4x_1}\cr 
		&\Sigma_{y_2y_1^3}&\rightarrow&\frac{1}{4}\Sigma_{y_3}\stuffle\Sigma_{y_2}+\frac{5}{4}\Sigma_{y_5}&S_{x_0x_1^4}&\rightarrow&S_{x_0^4x_1}\cr 
		\hline
		&\Sigma_{y_6}&\rightarrow&\frac{8}{35}\Sigma_{y_2}^{\stuffle 3}& S_{x_0^5x_1}&\rightarrow&\frac{8}{35}S_{x_0x_1}^{\shuffle 3}\cr 
		&\Sigma_{y_4y_2}&\rightarrow&\Sigma_{y_3}^{\stuffle 2}-\frac{4}{21}\Sigma_{y_2}^{\stuffle 3}& S_{x_0^4x_1^2}&\rightarrow&\frac{6}{35}S_{x_0x_1}^{\shuffle 3}-\frac{1}{2}S_{x_0^2x_1}^{\shuffle 2}\cr 
		&\Sigma_{y_5y_1}&\rightarrow&\frac{2}{7}\Sigma_{y_2}^{\stuffle 3}-\frac{1}{2}\Sigma_{y_3}^{\stuffle 2}& S_{x_0^3x_1x_0x_1}&\rightarrow&\frac{4}{105}S_{x_0x_1}^{\shuffle 3} \cr 
		&\Sigma_{y_3y_1y_2}&\rightarrow&-\frac{17}{30}\Sigma_{y_2}^{\stuffle 3}+\frac{9}{4}\Sigma_{y_3}^{\stuffle 2}&S_{x_0^3x_1^3}&\rightarrow&\frac{23}{70}S_{x_0x_1}^{\shuffle 3}-S_{x_0^2x_1}^{\shuffle 2}\cr 
		&\Sigma_{y_3y_2y_1}&\rightarrow&3\Sigma_{y_3}^{\stuffle 2}-\frac{9}{10}\Sigma_{y_2}^{\stuffle 3}& S_{x_0^2x_1x_0x_1^2}&\rightarrow&\frac{2}{105}S_{x_0x_1}^{\shuffle 3}\cr 
		6&\Sigma_{y_4y_1^2}&\rightarrow&\frac{3}{10}\Sigma_{y_2}^{\stuffle 3}-\frac{3}{4}\Sigma_{y_3}^{\stuffle 2}&S_{x_0^2x_1^2x_0x_1}&\rightarrow&-\frac{89}{210}S_{x_0x_1}^{\shuffle 3}+\frac{3}{2}S_{x_0^2x_1}^{\shuffle 2} \cr 
		&\Sigma_{y_2^2y_1^2}&\rightarrow&\frac{11}{63}\Sigma_{y_2}^{\stuffle 3}-\frac{1}{4}\Sigma_{y_3}^{\stuffle 2}& S_{x_0^2x_1^4}&\rightarrow&\frac{6}{35}S_{x_0x_1}^{\shuffle 3}-\frac{1}{2}S_{x_0^2x_1}^{\shuffle 2}\cr 
		&\Sigma_{y_3y_1^3}&\rightarrow&\frac{1}{21}\Sigma_{y_2}^{\stuffle 3}&S_{x_0x_1x_0x_1^3}&\rightarrow&\frac{8}{21}S_{x_0x_1}^{\shuffle 3}-S_{x_0^2x_1}^{\shuffle 2} \cr 
		&\Sigma_{y_2y_1^4}&\rightarrow&\frac{17}{50}\Sigma_{y_2}^{\stuffle 3}+\frac{3}{16}\Sigma_{y_3}^{\stuffle 2}&S_{x_0x_1^5}&\rightarrow&\frac{8}{35}S_{x_0x_1}^{\shuffle 3}\cr 
		\hline
	\end{array}$
\begin{eqnarray*}
\calL_{irr}^{X,\le12}&=&\{S_{x_0x_1},S_{x_0^2x_1},S_{x_0^4x_1},S_{x_0^6x_1},S_{x_0x_1^2x_0x_1^4},S_{x_0^8x_1},S_{x_0x_1^2x_0x_1^6},S_{x_0^{10}x_1},S_{x_0x_1^2x_0x_1^7},S_{x_0x_1^2x_0x_1^8},S_{x_0x_1^4x_0x_1^6}\}.\cr
\calL_{irr}^{Y,\le12}&=&\{\Sigma_{y_2},\Sigma_{y_3},\Sigma_{y_5},\Sigma_{y_7},\Sigma_{y_3y_1^5},\Sigma_{y_9},\Sigma_{y_3y_1^7},\Sigma_{y_{11}},\Sigma_{y_2y_1^9},\Sigma_{y_3y_1^9},\Sigma_{y_2^ 2y_1^8}\}.\cr
\end{eqnarray*}
\end{example}

With the notations introduced in \eqref{Zirr}--\eqref{Rirr}, let us return to some relevant results already mentioned in \cite{acta,VJM,CM}, for which the previous algorithm is the harvested fruit (briefly described in \cite{LT12,Tokyo}).

\begin{proposition}\label{directsum}
With the notation in \eqref{convenience}, one has
\begin{enumerate}
\item\label{un} $\mathrm{Im}\,\zeta=(\calZ,\times,1)\cong(\Q[{\calZ_{irr}^{\calX,\infty}}],\times,1)$ and $\calR_{\calX}=\ker\zeta$.

\item\label{deux} $\Q[\Lyn\calX\setminus\gDIV]/\calR_{\calX}\cong\Q[{\calL_{irr}^{\calX,\infty}}]$.

\item\label{trois} $\calZ_{irr}^{\calX,\le p}$ is $\Q$-algebraically independent, for $p\ge2$. So is $\calZ_{irr}^{\calX,\infty}$.

\item\label{quadre} Let $w$ in $CONV$. Then $\zeta(w)\in(\Q[{\calZ_{irr}^{\calX,\infty}}],\times,1)$.

\item\label{cinq} $\Q[\{S_l\}_{l\in\Lyn X\setminus X}]=\calR_X\oplus\Q[{\calL_{irr}^{X,\infty}}]$ and $\Q[\{\Sigma_l\}_{l\in\Lyn Y\setminus\{y_1\}}]=\calR_Y\oplus\Q[\calL_{irr}^{Y,\infty}]$.
\end{enumerate}
\end{proposition}

\begin{proof}
\begin{enumerate}
\item By Corollary \ref{zeta} and the constructions described in \eqref{restriction}--\eqref{R_XR_Y} (see also Remark \ref{bijection}), one obtains
\begin{eqnarray*}
\mathrm{Im}\,\zeta=(\calZ,\times,1)\cong(\Q[{\calZ_{irr}^{\calX,\infty}}],\times,1)
\end{eqnarray*}
and
\begin{eqnarray*}
\calR_{\calX}\subseteq\ker\zeta\subsetneq\Q[\Lyn\calX\setminus\gDIV].
\end{eqnarray*}

Now, if $Q\in\ker\zeta$ such that $\scal{Q}{1_{\calX^*}}=0$ then there is $Q_1\in{\mathcal R}_{\calX}$ and $Q_2\in\Q[\calL_{irr}^{\calX,\infty}]$ such that  $Q=Q_1+Q_2$. Thus, decomposing in the $\Q$-algebraic basis $\{S_l\}_{l\in\Lyn X\setminus X}$, or $\{\Sigma_l\}_{l\in\Lyn Y\setminus\{y_1\}}$, and then reducing by the rewriting rules in $\calR_{irr}^{\calX}$ and by hypothesis, one obtains.
\begin{eqnarray*}
Q\equiv_{\calR_{irr}^{\calX}}Q_1\in{\mathcal R}_{\calX}.
\end{eqnarray*}

Finally, using the restriction of the $\zeta$ polymorphism, described in \eqref{restriction} which is bijective, it follows then the expected result.

\item It is a consequence of Item \ref{un}.

\item Extracted from the $\Q$-algebraic basis $\{S_l\}_{l\in\Lyn X}$, or $\{\Sigma_l\}_{l\Lyn Y}$, the family $\calL_{irr}^{\calX,\le p}$ is $\Q$-algebraically independent. By the restriction of the $\zeta$ polymorphism, described in \eqref{restriction} which is bijective, the family $\calZ_{irr}^{\calX,\le p}$ is $\Q$-algebraically independent.

The proof is similar for $\calZ_{irr}^{\calX,\infty}$.

\item Since $\{S_l\}_{l\in\Lyn X\setminus X}$ (resp. $\{\Sigma_l\}_{l\in\Lyn Y\setminus\{y_1\}}$) is a pure transcendent basis of the $\Q$-algebra $(\Q1_{X^*}\oplus x_0\ncp{\Q}{X}x_1,\shuffle,1_{X^*})$ (resp. $(\Q1_{Y^*}\oplus(Y\setminus\{y_1\})\ncp{\Q}{Y},\stuffle,1_{Y^*})$) then, for any word $w$ in $x_0X^*x_1$ (resp. $(Y\setminus\{y_1\})Y^*$, after a decomposition in the $\Q$-algebraic basis $\{S_l\}_{l\in\Lyn X\setminus X}$ (resp. $\{\Sigma_l\}_{l\in\Lyn Y\setminus\{y_1\}}$), a reduction of $w$ by the rewriting rules in $\calR_{irr}^{\calX}$ is a polynomial belonging to $\Q[\calL_{irr}^{\calX,\infty}]$ and then, by the restriction of the $\zeta$ polymorphism, described in \eqref{restriction} which is bijective, the polyzeta $\zeta(w)$ belongs to $(\Q[\calZ_{irr}^{\calX,\infty}],\times,1)$.

\item Let $P\in(\Q[\{S_l\}_{l\in\Lyn X\setminus X}],\shuffle,1_{X^*})$, or $(\Q[\{\Sigma_l\}_{l\in\Lyn Y\setminus\{y_1\}}],\stuffle,1_{Y^*})$, such that $P\notin\ker\zeta\supseteq{\calR}_{\calX}$. Then, by Item \ref{deux} and by linearity, $\zeta(P)\in(\Q[\calZ_{irr}^{\calX,\infty}],\times,1)$. On the other hand, if $Q\in{\calR}_{\calX}\cap\Q[\calL_{irr}^{\calX,\infty}]$ then, by \eqref{R_XR_Y}, $\zeta(Q)=0$ and then, by \eqref{restriction} again, one obtains $Q=0$ leading to the expected result.
\end{enumerate}
\end{proof}

\begin{corollary}\label{transcendent}
The $\shuffle$-ideal $\calR_X$ (resp. $\stuffle$-ideal $\calR_Y$), obtaind by the main algorithm {\bf LocalCoordinateIdentification}, is isomorphic to the $\shuffle$-ideal (resp. $\stuffle$-ideal) $\Q$-alge\-braicly generated by the homogenous in weight polynomials $\{l_1\shuffle l_2-\pi_X(\pi_Y(l_1)\allowbreak\stuffle\allowbreak\pi_Y(l_2))\}_{l_1,l_2\in\Lyn X\setminus\{x_0\}}$ (resp. $\{l_1\stuffle l_2-\pi_Y(\pi_X(l_1)\shuffle\pi_X(l_2))\}_{l_1,l_2\in\Lyn Y}$).
\end{corollary}

\begin{proof}
In Corollary \ref{pont}, the identities $Z_{\gamma}=B(y_1)\pi_Y(Z_{\shuffle})$ and $Z_{\stuffle}=B'(y_1)\pi_Y(Z_{\shuffle})$, bridge the grouplike series $Z_{\shuffle}$ and $Z_{\gamma}$ (or $Z_{\stuffle}$), for $\Delta_{\shuffle}$ and $\Delta_{\shuffle}$, respectively.

On the one hand, by Proposition \ref{noyeaux}, the regularized double shuffle relations with simultaneous regularization to $0$ (see Remark \ref{regularizations}) are equivalent to the facts that $Z_{\shuffle}$ and $Z_{\stuffle}$ are grouplike. On the other hand, by Proposition \ref{directsum}, the $\shuffle$-ideal $\calR_X$ and $\stuffle$-ideal $\calR_Y$, as being the kernels of the $\zeta$ polymorphim, are furnished by these identities in Corollary \ref{pont}, via the main algorithm {\bf LocalCoordinateIdentification}. It follows then the expected results.

In fact, the factors $B'(y_1)$ (in $Z_{\stuffle}=B'(y_1)\pi_Y(Z_{\shuffle})$) and $B(y_1)$ (in $Z_{\gamma}=B(y_1)\pi_Y(Z_{\shuffle})$) are necessary to express, one the one hand, $\{\zeta_{\stuffle}(y_1^kw)\}_{k\ge1\atop w\in\CONV\cup\{1_{Y^*}\}}$ as following
\begin{eqnarray*}
\zeta_{\stuffle}(y_1^2)&=&-\frac12\zeta(2),\cr
\zeta_{\stuffle}(y_1^3)&=&\frac13\zeta(3),\cr
\zeta_{\stuffle}(y_1y_7)&=&\zeta(3)\zeta(5)-\frac{54}{175}\zeta(2)^4,\cr
\zeta_{\stuffle}(y_1^2y_6)&=&\frac{19}{35}\zeta(2)^4+\frac12\zeta(2)\zeta(3)^2+\zeta(6,2)-4\zeta(3)\zeta(5),
\end{eqnarray*}
and, on the other hand, $\{\gamma_{y_1^kw}\}_{k\ge1\atop w\in\CONV\cup\{1_{Y^*}\}}$ as in Example \ref{generalizedgamma} which cannot be obtained using the regularizations $\zeta^T_{\shuffle}$ and $\zeta^T_{\stuffle}$ described in the remark \ref{regularizations}, even with $T=\gamma$.
\end{proof}

Now, by the restriction of  the $\zeta$ polymorphism (described in \eqref{restriction} which is bijective), by Propositions \ref{noyeaux}--\ref{directsum} and Corollary \ref{transcendent}, we are in situation to define.

\begin{definition}\label{formalzetavalues}
\begin{enumerate}
\item The elements of $\Q[\calZ_{irr}^{\calX,\infty}]$ are called \textit{real polyzetas}.

\item The elements of $\Q[\calL_{irr}^{\calX,\infty}]$ are called \textit{formal polyzetas}. 
\end{enumerate}
\end{definition}

\begin{theorem}\label{graded}
The $\Q$-algebra $(\calZ,\times,1)$, being a subalgebra of $(\R,\times,1)$, is freely generated by $\calZ_{irr}^{\calX,\infty}$ and, as $\Q$-module,
\begin{eqnarray*}
\calZ=\Q1\oplus\bigoplus_{k\ge2}\calZ_k.
\end{eqnarray*}
\end{theorem}

\begin{proof}
By Proposition \ref{directsum}, the $\Q$-algebra $(\calZ,\times,1)$ is freely generated by $\calZ_{irr}^{\calX,\infty}$ and $\ker\zeta$, being generated by the homogeneous in weight polynomials $\{Q_l\}_{l\in\Lyn\calX\setminus\gDIV}$, is graded. With the notations in Conjecture \ref{Zconj}, $\calZ$ is also graded, as being $\mathrm{Im}\,\zeta$ which is isomorphic to $(\Q1_{Y^*}\oplus(Y\setminus\{y_1\})\ncp{\Q}{Y})/\ker\zeta$ and to $(\Q1_{X^*}\oplus x_0\ncp{\Q}{X}x_1/\ker\zeta$.
\end{proof}

\begin{corollary}\label{transcendent}
One has
\begin{enumerate}
\item Let $P\in\ncp{\Q}{\calX}$ and $P\notin\ker\zeta$, homogeneous in weight, or $P\in\CONV$. Then $\zeta(P)$ is a transcendental number over $\Q$.

\item Let $P\in\calL_{irr}^{\calX,\infty}$. Then $\zeta(P)$ is a transcendental number over $\Q$.

\item ${\zeta(2q+1)}/\pi^{2p}\notin\Q$ and $\zeta(2p+1)/\zeta(2q+1)\notin\Q$, for $p,q\ge1,p\neq q$.
\end{enumerate}
\end{corollary}

\begin{proof}
\begin{enumerate}
\item By \eqref{prod}, each monomial $(\zeta(P))^k$ ($k\ge1)$ is of different weight and then, by Theorem \ref{graded}, $\zeta(P)$ could not satisfy an algebraic equation $T^k+a_{k-1}T^{k-1}+\ldots=0$ with coefficients in $\Q$, \textit{i.e.} the expected result.

\item Any $P\in\calL_{irr}^{\calX,\infty}$ is homogeneous in weight and then, by the previous item, $\zeta(P)$ is a transcendental number over $\Q$.

\item $\zeta(2q+1)\in\calZ_{2q+1}$ and $\zeta(2p+1)\in\calZ_{2p+1}$ ($p\neq q$) and, by the rationality of $\zeta(2n)/\pi^{2n}$, one has $\pi^{2n}\in\calZ_{2p}$. By Theorem \ref{graded}, \textit{i.e.} there is no $\Q$-linear relation among elements of different component, it follows then the expected results.
\end{enumerate}
\end{proof}

\begin{corollary}\label{independence}
Let $\calP$ (resp. $\calQ$) be the $\Q$-algebra generated by $\calZ_{irr}^{\calX,\infty}\setminus\{\zeta(2)\}$ (resp. $\{\zeta(2q+1)\}_{q\ge1}$). Then
\begin{enumerate}
\item $\pi$ is transcendental over $\calP$.

\item If $\{\zeta(2q+1)\}_{q\ge1}\subset\calZ_{irr}^{\calX,\infty}$ then $\pi$ is transcendental over $\calQ$.
\end{enumerate}
\end{corollary}

\begin{proof}
One has, on the one hand, as said above $\zeta(2)=\pi^2/6$ is irreducible (see also Example \ref{E0}) and then, by Proposition \ref{directsum} (Item \ref{trois}), $\pi^2$ is transcendental over $\calP$. On the other hand, for any $p\ge1$, by Corollary \ref{transcendent},
\begin{eqnarray*}
\zeta(2p+1)\notin\bigoplus_{n\ge1}\Q[\pi^{2n}]
\end{eqnarray*}
meaning that $\zeta(2p+1)$ is a transcendental over $\Q[\pi^2]$ and then, by hypothesis and by Proposition \ref{directsum} (Item \ref{trois}), $\pi^2$ is transcendental over $\calQ$.

Hence, if there exists $P\in\calP[T]$ (resp. $\calQ[T]$) such that $P\neq 0$ and $P(\pi)=0$ then let $R\in\calP[T]$ be defined by
$$R(T):=P(T)P(-T).$$
It follows that $R(T)=R(-T)$ and, by hypothesis, $R(\pi)=0$ meaning that $\pi^2$ satisfies an algebraic equation with coefficients in $\calP$ (resp. $\calQ$) contradicting with two previous assertions. Thus, such $P$ does not exist and then $\pi$ is transcendental over $\calP$ (resp. $\calQ$).
\end{proof}

Ending this note, let $l\in\Lyn\calX$ such that $(l)=n>1$. Then
\begin{eqnarray}
\Sigma_l\succ\Sigma_{y_n}=y_n&\mbox{and}&S_l\succ S_{x_0^{n-1}x_1}=x_0^{n-1}x_1.
\end{eqnarray}
But unfortunately
\begin{eqnarray}
\{\Sigma_{y_n}\}_{n>1}\not\subset\calL_{irr}^{Y,\infty}&\mbox{and}&
\{S_{x_0^{n-1}x_1}\}_{n>1}\not\subset\calL_{irr}^{X,\infty}.
\end{eqnarray}

Indeed, $\zeta(\Sigma_{y_2})=\zeta(S_{x_0x_1})=\zeta(2)=\pi^2/6$ is irreducible then so are $\Sigma_{y_2}$ and $S_{x_0x_1}$. The rationality of $\zeta(2n)/\pi^{2n}$ (see \eqref{paire}) means that $\Sigma_{y_{2n}}=y_{2n}$ and $S_{x_0^{2n-1}x_1}=x_0^{2n-1}x_1$ are not irreducible, for $n>1$.
On the other hand, $\{\zeta(2n+1)\}_{1\le n\le 5}$, or equivalently $\{\Sigma_{2n+1}\}_{1\le n\le 5}$ and $\{S_{x_0^{2n}x_1}\}_{1\le n\le 5}$ , are irreducible (see Examples \ref{E1}--\ref{E0}). It could remain that $\{\zeta(2n+1)\}_{n>5}$, or equivalently $\{\Sigma_{2n+1}\}_{n>5}$ and $\{S_{x_0^{2n}x_1}\}_{n>5}$, could be irreducible.

\section{Conclusion}
Thanks to an Abel like theorem (see Theorem \ref{renormalization1}) and the equations bridging the algebraic structures (see Corollary \ref{pont}) of the $\Q$-algebra $(\calZ,\times,1)$, being a subalgebra of $(\R,\times,1)$, generated by the polyzetas $\{\zeta(w)\}_{w\in\CONV}$ (see the notations in \eqref{convenience}), the $\Q$-algebraic relations among the local coordinates (of second kind on the groups of grouplike series) of the grouplike series $Z_{\shuffle}$ (\textit{i.e.} among the polyzetas $\{\zeta(S_l)\}_{l\in\Lyn X\setminus X}$) and of $Z_{\stuffle}$ (\textit{i.e.} among the polyzetas $\{\zeta(\Sigma_l)\}_{l\in\Lyn Y\setminus\{y_1\}}$) are provided by the main algorithm {\bf LocalCoordinateIdentification}.

These $\Q$-algebraic relations are interpreted as two confluent, without critical pairs rewriting systems $(\Q1_{Y^*}\oplus(Y\setminus\{y_1\})\ncp{\Q}{Y},\calR^Y_{irr})$ and $(\Q1_{X^*}\oplus x_0\ncp{\Q}{X}x_1,\calR^X_{irr})$ admitting $\calL_{irr}^{\calX,\infty}$ as the set of irreducible terms, in which the irreducible terms represent the $\Q$-algebraic generators for $(\calZ,\times,1)$ and, on the other hand, the $\shuffle$-ideal $\calR_X$ and the $\stuffle$-ideal $\calR_Y$ represent the kernels of the $\zeta$ polymorphism (see Proposition \ref{directsum}).

These $\Q$-ideals are generated by the polynomials, totally ordered and homogeneous in weight, $\{Q_l\}_{l\in\Lyn\calX\setminus\gDIV}$, and are interpreted as the confluent rewriting systems in which the irreducible terms belong to $\calL_{irr}^{\calX,\infty}$ and, in each rewriting rule of $\calR_{irr}^{\calX}$, the left side is the leading monomial of $Q_l$ ($l\in\Lyn\calX\setminus\gDIV$) and does not belong to $\Q[\calL_{irr}^{\calX,\infty}]$ while the right side is canonically represented on the $\Q$-algebra $\Q[\calL_{irr}^{\calX,\infty}]$.

It follows that the algebra $(\Q[\{\zeta(l)\}_{l\in\calL_{irr}^{\calX,\infty}}],\times,1)$, \textit{i.e.} $(\calZ,\times,1)$, is isomorphic to the quotien $(\Q1_{X^*}\oplus x_0\ncp{\Q}{X}x_1,\shuffle,1_{X^*})/\calR_X$, or $(\Q1_{Y^*}\oplus(Y\setminus\{y_1\})\ncp{\Q}{Y},\stuffle,1_{Y^*})/\calR_Y$. It is $\Q$-free and graded (see Theorem \ref{graded}) and then the irreducible polyzetas, belonging to $\calZ_{irr}^{\calX,\infty}$ and being $\Q$-algebraic free, are transcendental numbers (see Corollary \ref{transcendent}).

Up to weight $12$, by these results, Conjecture \ref{Zconj} holds, \textit{i.e.} $\calZ_{irr}^{\calX,\le12}$ is $\Q$-algebraically free (see Example \ref{E} and see \cite{VJM} for a short discussion, see also \cite{racinet,kaneko}\footnote{All these implementations base on the ``double shuffle relations" and provide linear relations.}). In particular, $\pi^2$ is $\Q$-algebraically free on $\{\zeta(2q+1)\}_{1\le q\le5}$ values. So is $\pi$ (see Corollary \ref{independence}).

\end{document}